\documentclass[11pt,twoside,a4paper]{article}

\usepackage{color,amscd,amsmath,amssymb,amsthm,latexsym,stmaryrd}
\usepackage{shuffle}
\usepackage[latin1]{inputenc}
\usepackage[T1]{fontenc}   
\usepackage[english]{babel}
\usepackage{tikz,tkz-tab}
\usepackage{float}
\usepackage{mathtools}
\usepackage{mathabx}
\usepackage{subcaption}
\usepackage{bbm}

\usepackage{todonotes}

\newcommand{\tun}{\begin{tikzpicture}[line cap=round,line join=round,>=triangle 45,x=0.5cm,y=0.5cm]
\clip(-0.2,-0.1) rectangle (0.2,0.2);
\begin{scriptsize}
\draw [fill=black] (0.,0.) circle (1pt);
\end{scriptsize}
\end{tikzpicture}}

\newcommand{\tdeux}{\begin{tikzpicture}[line cap=round,line join=round,>=triangle 45,x=0.5cm,y=0.5cm]
\clip(-.2,-.1) rectangle (0.2,0.7);
\draw [line width=.5pt] (0.,0.5)-- (0.,0.);
\begin{scriptsize}
\draw [fill=black] (0.,0.) circle (1pt);
\draw [fill=black] (0.,0.5) circle (1pt);
\end{scriptsize}
\end{tikzpicture}}

\newcommand{\ttroisun}{\begin{tikzpicture}[line cap=round,line join=round,>=triangle 45,x=0.5cm,y=0.5cm]
\clip(-0.5,-0.1) rectangle (0.5,0.7);
\draw [line width=0.5pt] (0.,0.)-- (-0.3,0.5);
\draw [line width=0.5pt] (0.,0.)-- (0.3,0.5);
\begin{scriptsize}
\draw [fill=black] (-0.3,0.5) circle (1pt);
\draw [fill=black] (0.,0.) circle (1pt);
\draw [fill=black] (0.3,0.5) circle (1pt);
\end{scriptsize}
\end{tikzpicture}}
\newcommand{\ttroisdeux}{\begin{tikzpicture}[line cap=round,line join=round,>=triangle 45,x=0.5cm,y=0.5cm]
\clip(-.2,-.1) rectangle (0.2,1.2);
\draw [line width=0.5pt] (0.,0.5)-- (0.,0.);
\draw [line width=0.5pt] (0.,0.5)-- (0.,1.);
\begin{scriptsize}
\draw [fill=black] (0.,0.) circle (1pt);
\draw [fill=black] (0.,0.5) circle (1pt);
\draw [fill=black] (0.,1.) circle (1pt);
\end{scriptsize}
\end{tikzpicture}}

\newcommand{\tquatreun}{\begin{tikzpicture}[line cap=round,line join=round,>=triangle 45,x=0.5cm,y=0.5cm]
\clip(-0.5,-0.1) rectangle (0.5,0.7);
\draw [line width=0.5pt] (0.,0.)-- (-0.3,0.5);
\draw [line width=0.5pt] (0.,0.)-- (0.3,0.5);
\draw [line width=0.5pt] (0.,0.)-- (0.,0.5);
\begin{scriptsize}
\draw [fill=black] (-0.3,0.5) circle (1.0pt);
\draw [fill=black] (0.,0.) circle (1.0pt);
\draw [fill=black] (0.3,0.5) circle (1.0pt);
\draw [fill=black] (0.,0.5) circle (1.0pt);
\end{scriptsize}
\end{tikzpicture}}
\newcommand{\tquatredeux}{\begin{tikzpicture}[line cap=round,line join=round,>=triangle 45,x=0.5cm,y=0.5cm]
\clip(-0.5,-0.1) rectangle (0.5,1.2);
\draw [line width=0.5pt] (0.,0.)-- (-0.3,0.5);
\draw [line width=0.5pt] (0.,0.)-- (0.3,0.5);
\draw [line width=0.5pt] (-0.3,0.5)-- (-0.3,1.);
\begin{scriptsize}
\draw [fill=black] (-0.3,0.5) circle (1.0pt);
\draw [fill=black] (0.,0.) circle (1.0pt);
\draw [fill=black] (0.3,0.5) circle (1.0pt);
\draw [fill=black] (-0.3,1.) circle (1.0pt);
\end{scriptsize}
\end{tikzpicture}}

\newcommand{\tquatrequatre}{\begin{tikzpicture}[line cap=round,line join=round,>=triangle 45,x=0.5cm,y=0.5cm]
\clip(-0.5,-0.1) rectangle (0.5,1.7);
\draw [line width=0.5pt] (0.,0.)-- (0.,0.5);
\draw [line width=0.5pt] (0.,0.5)-- (0.3,1.);
\draw [line width=0.5pt] (0.,0.5)-- (-0.3,1.);
\begin{scriptsize}
\draw [fill=black] (0.,0.) circle (1.0pt);
\draw [fill=black] (0.,0.5) circle (1.0pt);
\draw [fill=black] (-0.3,1.) circle (1.0pt);
\draw [fill=black] (0.3,1.) circle (1.0pt);
\end{scriptsize}
\end{tikzpicture}}
\newcommand{\tquatrecinq}{\begin{tikzpicture}[line cap=round,line join=round,>=triangle 45,x=0.5cm,y=0.5cm]
\clip(-.2,-.1) rectangle (0.5,1.7);
\draw [line width=0.5pt] (0.,0.)-- (0.,0.5);
\draw [line width=0.5pt] (0.,0.5)-- (0.,1.);
\draw [line width=0.5pt] (0.,1.)-- (0.,1.5);
\begin{scriptsize}
\draw [fill=black] (0.,0.) circle (1.0pt);
\draw [fill=black] (0.,0.5) circle (1.0pt);
\draw [fill=black] (0.,1.) circle (1.0pt);
\draw [fill=black] (0.,1.5) circle (1.0pt);
\end{scriptsize}
\end{tikzpicture}}

\newcommand{\tcinqun}{\begin{tikzpicture}[line cap=round,line join=round,>=triangle 45,x=0.5cm,y=0.5cm]
\clip(-0.7,-0.1) rectangle (0.8,0.7);
\draw [line width=0.5pt] (0.,0.)-- (-0.5,0.5);
\draw [line width=0.5pt] (0.,0.)-- (-0.2,0.5);
\draw [line width=0.5pt] (0.,0.)-- (0.2,0.5);
\draw [line width=0.5pt] (0.,0.)-- (0.5,0.5);
\begin{scriptsize}
\draw [fill=black] (-0.5,0.5) circle (1.0pt);
\draw [fill=black] (-0.2,0.5) circle (1.0pt);
\draw [fill=black] (0.2,0.5) circle (1.0pt);
\draw [fill=black] (0.5,0.5) circle (1.0pt);
\draw [fill=black] (0.,0.) circle (1.0pt);
\end{scriptsize}
\end{tikzpicture}}
\newcommand{\tcinqdeux}{\begin{tikzpicture}[line cap=round,line join=round,>=triangle 45,x=0.5cm,y=0.5cm]
\clip(-0.5,-0.1) rectangle (0.5,1.2);
\draw [line width=0.5pt] (0.,0.)-- (-0.3,0.5);
\draw [line width=0.5pt] (0.,0.)-- (0.3,0.5);
\draw [line width=0.5pt] (0.,0.)-- (0.,0.5);
\draw [line width=0.5pt] (-0.3,0.5)-- (-0.3,1.);
\begin{scriptsize}
\draw [fill=black] (-0.3,0.5) circle (1.0pt);
\draw [fill=black] (0.,0.) circle (1.0pt);
\draw [fill=black] (0.,0.5) circle (1.0pt);
\draw [fill=black] (0.3,0.5) circle (1.0pt);
\draw [fill=black] (-0.3,1.) circle (1.0pt);
\end{scriptsize}
\end{tikzpicture}}

\newcommand{\tcinqcinq}{\begin{tikzpicture}[line cap=round,line join=round,>=triangle 45,x=0.5cm,y=0.5cm]
\clip(-0.5,-0.1) rectangle (0.5,1.2);
\draw [line width=0.5pt] (0.,0.)-- (-0.3,0.5);
\draw [line width=0.5pt] (-0.3,0.5)-- (-0.3,1.);
\draw [line width=0.5pt] (0.,0.)-- (0.3,0.5);
\draw [line width=0.5pt] (0.3,0.5)-- (0.3,1.);
\begin{scriptsize}
\draw [fill=black] (-0.3,0.5) circle (1.0pt);
\draw [fill=black] (0.,0.) circle (1.0pt);
\draw [fill=black] (-0.3,1.) circle (1.0pt);
\draw [fill=black] (0.3,0.5) circle (1.0pt);
\draw [fill=black] (0.3,1.) circle (1.0pt);
\end{scriptsize}
\end{tikzpicture}}
\newcommand{\tcinqsix}{\begin{tikzpicture}[line cap=round,line join=round,>=triangle 45,x=0.5cm,y=0.5cm]
\clip(-0.7,-0.1) rectangle (0.5,1.2);
\draw [line width=0.5pt] (0.,0.)-- (-0.3,0.5);
\draw [line width=0.5pt] (0.,0.)-- (0.3,0.5);
\draw [line width=0.5pt] (-0.3,0.5)-- (-0.6,1.);
\draw [line width=0.5pt] (-0.3,0.5)-- (0.,1.);
\begin{scriptsize}
\draw [fill=black] (-0.3,0.5) circle (1.0pt);
\draw [fill=black] (0.,0.) circle (1.0pt);
\draw [fill=black] (0.3,0.5) circle (1.0pt);
\draw [fill=black] (-0.6,1.) circle (1.0pt);
\draw [fill=black] (0.,1.) circle (1.0pt);
\end{scriptsize}
\end{tikzpicture}}

\newcommand{\tcinqhuit}{\begin{tikzpicture}[line cap=round,line join=round,>=triangle 45,x=0.5cm,y=0.5cm]
\clip(-0.5,-0.1) rectangle (0.5,1.7);
\draw [line width=0.5pt] (0.,0.)-- (-0.3,0.5);
\draw [line width=0.5pt] (-0.3,0.5)-- (-0.3,1.);
\draw [line width=0.5pt] (0.,0.)-- (0.3,0.5);
\draw [line width=0.5pt] (-0.3,1.)-- (-0.3,1.5);
\begin{scriptsize}
\draw [fill=black] (-0.3,0.5) circle (1.0pt);
\draw [fill=black] (0.,0.) circle (1.0pt);
\draw [fill=black] (-0.3,1.) circle (1.0pt);
\draw [fill=black] (0.3,0.5) circle (1.0pt);
\draw [fill=black] (-0.3,1.5) circle (1.0pt);
\end{scriptsize}
\end{tikzpicture}}

\newcommand{\tcinqdix}{\begin{tikzpicture}[line cap=round,line join=round,>=triangle 45,x=0.5cm,y=0.5cm]
\clip(-0.5,-0.1) rectangle (0.5,1.7);
\draw [line width=0.5pt] (0.,0.)-- (0.,0.5);
\draw [line width=0.5pt] (0.,0.5)-- (0.3,1.);
\draw [line width=0.5pt] (0.,0.5)-- (-0.3,1.);
\draw [line width=0.5pt] (0.,0.5)-- (0.,1.);
\begin{scriptsize}
\draw [fill=black] (0.,0.) circle (1.0pt);
\draw [fill=black] (0.,0.5) circle (1.0pt);
\draw [fill=black] (-0.3,1.) circle (1.0pt);
\draw [fill=black] (0.3,1.) circle (1.0pt);
\draw [fill=black] (0.,1.) circle (1.0pt);
\end{scriptsize}
\end{tikzpicture}}
\newcommand{\tcinqonze}{\begin{tikzpicture}[line cap=round,line join=round,>=triangle 45,x=0.5cm,y=0.5cm]
\clip(-0.5,-0.1) rectangle (0.5,1.7);
\draw [line width=0.5pt] (0.,0.)-- (0.,0.5);
\draw [line width=0.5pt] (0.,0.5)-- (0.3,1.);
\draw [line width=0.5pt] (0.,0.5)-- (-0.3,1.);
\draw [line width=0.5pt] (-0.3,1.)-- (-0.3,1.5);
\begin{scriptsize}
\draw [fill=black] (0.,0.) circle (1.0pt);
\draw [fill=black] (0.,0.5) circle (1.0pt);
\draw [fill=black] (-0.3,1.) circle (1.0pt);
\draw [fill=black] (0.3,1.) circle (1.0pt);
\draw [fill=black] (-0.3,1.5) circle (1.0pt);
\end{scriptsize}
\end{tikzpicture}}
\newcommand{\tcinqdouze}{\begin{tikzpicture}[line cap=round,line join=round,>=triangle 45,x=0.5cm,y=0.5cm]
\clip(-0.5,-0.1) rectangle (0.5,1.7);
\draw [line width=0.5pt] (0.,0.)-- (0.,0.5);
\draw [line width=0.5pt] (0.,0.5)-- (0.3,1.);
\draw [line width=0.5pt] (0.,0.5)-- (-0.3,1.);
\draw [line width=0.5pt] (0.3,1.)-- (0.3,1.5);
\begin{scriptsize}
\draw [fill=black] (0.,0.) circle (1.0pt);
\draw [fill=black] (0.,0.5) circle (1.0pt);
\draw [fill=black] (-0.3,1.) circle (1.0pt);
\draw [fill=black] (0.3,1.) circle (1.0pt);
\draw [fill=black] (0.3,1.5) circle (1.0pt);
\end{scriptsize}
\end{tikzpicture}}
\newcommand{\tcinqtreize}{\begin{tikzpicture}[line cap=round,line join=round,>=triangle 45,x=0.5cm,y=0.5cm]
\clip(-0.5,-0.1) rectangle (0.5,1.7);
\draw [line width=0.5pt] (0.,0.)-- (0.,0.5);
\draw [line width=0.5pt] (0.,0.5)-- (0.,1.);
\draw [line width=0.5pt] (0.,1.)-- (-0.3,1.5);
\draw [line width=0.5pt] (0.,1.)-- (0.3,1.5);
\begin{scriptsize}
\draw [fill=black] (0.,0.) circle (1.0pt);
\draw [fill=black] (0.,0.5) circle (1.0pt);
\draw [fill=black] (0.,1.) circle (1.0pt);
\draw [fill=black] (-0.3,1.5) circle (1.0pt);
\draw [fill=black] (0.3,1.5) circle (1.0pt);
\end{scriptsize}
\end{tikzpicture}}
\newcommand{\tcinqquatorze}{\begin{tikzpicture}[line cap=round,line join=round,>=triangle 45,x=0.5cm,y=0.5cm]
\clip(-.2,-.1) rectangle (0.5,2.2);
\draw [line width=0.5pt] (0.,0.)-- (0.,0.5);
\draw [line width=0.5pt] (0.,0.5)-- (0.,1.);
\draw [line width=0.5pt] (0.,1.)-- (0.,1.5);
\draw [line width=0.5pt] (0.,1.5)-- (0.,2.);
\begin{scriptsize}
\draw [fill=black] (0.,0.) circle (1.0pt);
\draw [fill=black] (0.,0.5) circle (1.0pt);
\draw [fill=black] (0.,1.) circle (1.0pt);
\draw [fill=black] (0.,1.5) circle (1.0pt);
\draw [fill=black] (0.,2.) circle (1.0pt);
\end{scriptsize}
\end{tikzpicture}}

\newcommand{\tquatredeuxAlternate}{\begin{tikzpicture}[line cap=round,line join=round,>=triangle 45,x=0.5cm,y=0.5cm]
\clip(-0.5,-0.1) rectangle (0.5,1.2);
\draw [line width=0.5pt] (0.,0.)-- (-0.3,0.5);
\draw [line width=0.5pt] (0.,0.)-- (0.3,0.5);
\draw [line width=0.5pt] (0.3,0.5)-- (0.3,1.);
\begin{scriptsize}
\draw [fill=black] (-0.3,0.5) circle (1.0pt);
\draw [fill=black] (0.,0.) circle (1.0pt);
\draw [fill=black] (0.3,0.5) circle (1.0pt);
\draw [fill=black] (0.3,1.) circle (1.0pt);
\end{scriptsize}
\end{tikzpicture}}

\input{xy}
\xyoption{all}

\definecolor{red}{rgb}{1.,0.,0.}

\theoremstyle{plain}
\newtheorem{theo}{Theorem}
\newtheorem{lemma}[theo]{Lemma}

\newtheorem{prop}[theo]{Proposition}
\newtheorem{defi}[theo]{Definition}

\theoremstyle{remark}
\newtheorem{remark}{Remark}

\newtheorem{example}{Example}

\newcommand{\K}{\mathbb{K}}
\newcommand{\N}{\mathbb{N}}
\newcommand{\seq}{\mathcal{S}eq}
\newcommand{\g}{\mathfrak{g}}

\newcommand{\lam}{\lambda}
\newcommand{\C}{\mathcal{C}}
\newcommand{\set}[1]{\{  #1  \}}
\newcommand{\cop}{\Delta}
\newcommand{\tensor}{\otimes}

\newcommand{\squareparens}[1]{\left[ #1 \right]}
\newcommand{\One}{\mathbbm{1}} 

\title{Sequences of Trees and Higher-Order Renormalization Group Equations}
\author{William T. Dugan, Lo\"\i c Foissy, and Karen Yeats}

\date{\today}

\begin{document}

\maketitle

\begin{abstract}
  We define a notion of higher order renormalization group equation and investigate when a sequence of trees satisfies such an equation.  In the strongest sense, the sequence of trees satisfies a $k$th order renormalization group equation when applying any choice of Feynman rules results in a Green function satisfying a $k$th order renormalization group equation, and we characterize all such sequences of trees.  We also make some comments on sequences of trees which require special choices of Feynman rules in order to satisfy a higher order renormalization group equation.  
\end{abstract}

\tableofcontents

\section{Introduction}
\label{sect::Introduction}

The renormalization group equation is a very important equation in quantum field theory since it describes how $n$-point functions of the theory change with changes in the energy scale and the coupling.  Such a description makes the renormalization group equation appear to be primarily in the domain of analysis and physics.  However, using the Hopf algebra structure of renormalization, the renormalization group equation can be seen as encoding a purely combinatorial property, from which the physics and analysis follow.  This combinatorial property is a particular kind of linear growth condition on the combinatorial objects which give the $n$-point functions; see Theorem \ref{theo1}.  From this linear growth condition, we are led to ask about polynomial growth of higher degree.  Translating back to the original renormalization group equation this yields higher order derivatives in the coupling.  The resulting equations we call \textbf{\emph{higher order renormalization group equations}}.  The first order case corresponds to the usual renormalization group equation and the zeroth order case corresponds to the special case in quantum field theory where the $\beta$-function of the theory is identically 0 and so the theory has a pure scaling solution.

For this paper we will work with the Connes-Kreimer Hopf algebra of rooted trees.  This Hopf algebra is universal among pairs of a Hopf algebra and a Hochschild 1-cocycle \cite{ConnesKreimer} so there is no great loss in this specialization.  For the Connes-Kreimer Hopf algebra the relevant 1-cocycle is $B_+$, the add-a-root operator, see Definition \ref{def::B+}.  The 1-cocycle provides the link to the Feynman rules, and hence to the renormalization group equation and the physics more generally, as we will see in Section \ref{sect::trees_and_tree_feynman_rules}.

Our goal, then, in this paper is to characterize Hopf subalgebras of the Connes-Kreimer Hopf algebra which satisfy higher order renormalization group equations in a way which is suitably insensitive to the choice of Feynman rules.  To pin down the scaling freedom, we will work not with Hopf subalgebras themselves but rather with the sequence of their generators.  This leads to our definition of a \textbf{\emph{strong $k$th order sequence}}, see Definition \ref{def::strong_sequence}.  We will characterize all strong $k$th order sequences, see Section \ref{subsect::strong_0th_order_seq} for zeroth order, Section \ref{subsect::strong_1st_order_seq} for first order, and Section \ref{subsect::2nd_order_seq} for higher order.  In particular we find the only strong sequences of order 2 or larger are scaled corollas.  This lines up with how quantum field theory only sees the zeroth and first order cases.  The characterization of zeroth and first order solutions includes some we recognize from physics or combinatorics, and some that appear new.

If we relax the condition that the sequence should satisfy a higher order renormalization group equation in a way which is insensitive to the choice of Feynman rules, and allow ourselves to make convenient choices of Feynman rules, then we get the notion of \textbf{\emph{weak $k$th order sequences}}, see Definition \ref{def::weak_sequence}.  These are much wilder, so we will not give a classification, but we will make some comments on a few examples of particular interest, see Section \ref{sect::comments_on_weak_sequences}.

To this end, we will proceed by laying out the relevant background and set up in Section~\ref{sect::set up}, then define our notion of higher order renormalization group equation and its combinatorial analogue in Section~\ref{sect::higher_order_RGEs}.  We will define our key notions of sequence and $\Lambda$-array in Section~\ref{sec lambda} as well as the notion of strong and weak $k$th order sequences.  Then we proceed to the characterization of strong sequences in Section~\ref{sect::characterization_of_strong_sequences} and conclude with some comments on weak sequences in Section~\ref{sect::comments_on_weak_sequences}.

\section{Background and set up}\label{sect::set up}

Let $\K$ be a field of characteristic zero. All the objects of this paper are taken over $\K$.

\subsection{Trees and tree Feynman rules}

\label{sect::trees_and_tree_feynman_rules}

Let $\mathcal{T}$ be the set of non-empty rooted trees.  Elements of $\mathcal{T}$ have no plane structure, so for example $$\tquatredeux = \tquatredeuxAlternate.$$

An \textbf{\emph{admissible cut}} of a tree $t \in \mathcal{T}$ is a non-empty subset $c$ of the edges of $t$ such that $c$ does not contain any two edges that lie on the same path from the root of $T$ to any leaf. See Figure~\ref{fig::admissible_cut}.

We define $H_{CK}$ the \textbf{\emph{Connes-Kreimer Hopf algebra of rooted trees}} as follows.  As an algebra $H_{CK} = \K[\mathcal{T}]$ where we view a forest as a monomial by identifying disjoint union of trees with the multiplication of the polynomial algebra and so the empty forest, notated $\One$, is the unit element of the algebra.  $H_{CK}$ is upgraded to a Hopf algebra via the following coproduct:
for any rooted tree $t\in \mathcal{T}$,
\[\Delta(t)= t \otimes \One + \One \otimes t + \sum_{\substack{\text{$c$ non-empty}\\\text{admissible cut of $t$}}}  P^c(t) \otimes R^c(t),\]
where $R^c(t)$ is the unique subtree containing the root of $t$ after removing the edges of $c$ and $P^c(t)$ is the forest of all trees other than $R^c(t)$ produced by removing the edges of $c$. We then extend $\Delta$ as an algebra homomorphism. The counit is $\eta(t) = 0$ for $t\in \mathcal{T}$, $\eta(\One)=1$ and extended as an algebra homomorphism.

\begin{figure}
    \centering
    \begin{subfigure}[b]{0.4\textwidth}
         \centering
            \includegraphics[width=2.5cm]{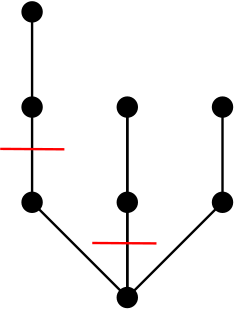}
         \caption{An admissible cut}
         \label{fig:y equals x}

    \end{subfigure}
    \begin{subfigure}[b]{0.5\textwidth}
         \centering
        \includegraphics[width=2.5cm]{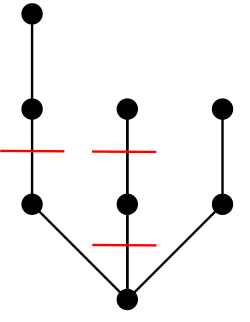}
\caption{A cut that is not admissible}
    \end{subfigure}
    \caption{Example and non-example of an of admissible cut. Edges with a red strike-through are those being included in the cut $c$. The cut in (b) is not admissible since the unique path from the root to the middle leaf is being cut twice.}
    \label{fig::admissible_cut}
\end{figure}

$H_{CK}$ is graded by the number of vertices and is connected under this grading, so $H_{CK}$ is a Hopf algebra with the antipode given recursively (see \cite{Sorin_D_Hopf_Algebras_An_Introduction}).  Specifically, the antipode is defined by
\[
S(t)= -t  - \sum_{\substack{\text{$c$ non-empty}\\\text{admissible cut of $t$}}}  S(P^c(t)) R^c(t)
\]
on trees and extends in general as an antiautomorphism, but our multiplication is commutative, so our $S$ is an automorphism.

\bigskip

We will be primarily interested in Hopf subalgebras with one generator in each degree.

\begin{example}\label{eg ladder 1}
Let $\ell_n$ be the rooted tree consisting of a single path with $n$ vertices and the root at one end, so \[
\ell_1 = \tun \quad \ell_2 = \tdeux \quad \ell_3 = \ttroisdeux.
\]
These are known as \textbf{\textit{ladder trees}}.  We can quickly check that the polynomial algebra $\K[\ell_1, \ell_2, \ldots]$ is in fact a Hopf subalgebra of $H_{CK}$ by noting that $\displaystyle \Delta(\ell_i) = \sum_{k=0}^i \ell_k\otimes \ell_{i-k}$ using the convention that $\ell_0=\One$.
\end{example}

\begin{example}
\label{ex::Connes_Moscovici}
Another important example of a Hopf subalgebra with one generator in each degree is the Connes-Moscovici Hopf algebra, which for our purposes is best defined as follows.  Say $\bullet$ is grafted on vertex $s$ of tree $t$ if $\bullet$ is made a new leaf of $t$ with parent $s$ and the tree is otherwise unchanged.
  Let us consider the growth operator $N:H_{CK}\longrightarrow H_{CK}$ defined by:
\[N(t)=\sum_{s\in Vert(t)}N_s(t),\]
where $Vert(t)$ denotes the set of vertices of $t$ and $N_s(t)$ denotes the tree resulting from grafting $\bullet$ onto $t$ as a new child of the vertex $s$ of $t$.
For example:
\begin{align*}
N(\tun)&=\tdeux,&N(\tdeux)&=\ttroisun+\ttroisdeux,&
N(\ttroisun)&=\tquatreun+2\tquatredeux,&N(\ttroisdeux)&=\tquatredeux+\tquatrequatre+\tquatrecinq.
\end{align*}
Then define a sequence $(t_n)_{n\geq 1}$ by $t_1=\tun$, $t_{n+1}=N(t_n)$ if $n\geq 1$.
These are the generators of the Connes-Moscovici subalgebra \cite{ConnesKreimer}. The sequence begins:
\begin{align*}
    t_1 & = \tun \\ 
    t_2 &= \tdeux \\ 
    t_3 &= \ttroisdeux + \ttroisun\\
    t_4 &= \tquatrecinq + \tquatrequatre + 3\tquatredeux + \tquatreun\\
    t_5 &= \tcinqquatorze + \tcinqtreize + 3\tcinqonze + \tcinqdix + 4\tcinqhuit + 4\tcinqsix+ 3\tcinqcinq + 6\tcinqdeux + \tcinqun
\end{align*}

\end{example}

Returning to the general situation, the add-a-root operator is very important, particularly for Dyson-Schwinger equations, and is defined as follows.
\begin{defi}
\label{def::B+}
  The linear map $B_+:H_{CK}\rightarrow H_{CK}$ is defined on forest $t_1t_2 \cdots t_k$ to be the tree $B_+(t_1t_2 \cdots t_k)$ with a new root whose children are the roots of $t_1, t_2, \ldots, t_k$.
\end{defi}
We have  

\begin{equation}
\label{eq::one_cocycle_property}
\Delta(B^+(t_1t_2\cdots t_k)) = B^+(t_1t_2\cdots t_k) \tensor \One + (Id \tensor B^+) \circ \Delta(t_1t_2\cdots t_k)
\end{equation}
and hence $B_+$ is a Hochschild 1-cocycle \cite{ConnesKreimer}.

\begin{example}\label{eg ladder 2}
Continuing Example~\ref{eg ladder 1}, we see that the class of ladder trees can be defined by $\ell_1 = \tun$, $\ell_{n+1} = B_+(\ell_n)$ for $n\geq 1$.
\end{example}

If $C$ is a coalgebra with coproduct $\Delta$ and $A$ is an algebra with product $m$ then we have the convolution product of two maps $f,g:C\rightarrow A$ given by $f * g = m\circ (f\otimes g)\circ \Delta$.  The important application for us is in the algebraic definition of Feynman rules.

\begin{defi}
\label{def::Feynman_rules}
\textbf{Feynman rules} are a Hopf algebra morphism $\phi : H_{CK} \rightarrow \K[L]$, where the coalgebra structure of $\K[L]$ is determined by $\cop(L) = L \tensor 1 + 1 \tensor L$ and extended linearly.
\end{defi}

{}From Lie theory, every such map $\phi$ can be written as an exponentiation of an infinitesimal character.  Explicitly, this gives us two very concrete consequences on $\phi$.  First we can write $\phi$ in terms of a nice convolution property.
\begin{lemma}
\label{lem::algebraic_RGE}
Write $\phi_{L'}$ for the Feynman rules $\phi$ with $L'$ as the variable in the target algebra in place of $L$.  Then 
\begin{equation}
\label{eq::algebraic_RGE}
    \phi_{L_1} * \phi_{L_2} = \phi_{L_1 + L_2}
\end{equation}
where $*$ is the convolution product. 
\end{lemma}
This can be thought of as the renormalization group equation in algebraic form, or as an alternate definition of the Feynman rules.  See Section A.4 of \cite{Lskript}.

Second, we can give an explicit form for all Feynman rules defined on trees.  First we need to define the tree factorial.  Given a tree $t$ and a vertex $v$ of $t$, let $t_v$ be the subtree rooted at $t$.  Then the tree factorial $t! = \prod_{v\in t} |t_v|$ where $|\cdot|$ is the number of vertices.  For example ladder trees give the usual factorial $\ell_n! = n!$.  The tree factorial of a forest is defined the same way and consequently is also the product of the tree factorials of its component trees.

\begin{lemma}
\label{lem::explicit_feynman_rules}
Let $\sigma:H_{CK}\rightarrow \K$ be an infinitesimal character.
Then
\begin{equation}
    \label{eq::explicit_tree_feynman_rules}
    \phi(F) = \sum_{S \subseteq E(F)}\bigg(\prod_{t\in (F \setminus S)}\sigma(t)\bigg) \frac{L^{|F \slash (F \setminus S)|}}{(F \slash (F \setminus S))!}
\end{equation}
defines Feynman rules $\phi$ and every choice of Feynman rules $\phi$ has this form where $\sigma$ is the infinitesimal character such that $\phi = \exp_*(L\sigma)$.
In the sum $E(F)$ is the edge set of $F$, $F \setminus S$ is the forest whose vertices are those of $F$ with the edges in $S$ removed, and $(F \slash (F \setminus S))$ means the forest $F$ with the edges of  $F \setminus S$ contracted.
The symbol $\displaystyle \prod_{t\in (F \setminus S)}$ means that the product is taken over all trees
of the forest $F\setminus S$.
\end{lemma}

Note here that when an edge is contracted then its two incident vertices are identified, so contracting an edge of a tree results in a tree with one vertex less, while when an edge is removed, it is deleted without any further change to its incident vertices, so removing an edge of a tree results in a forest of two trees.

\begin{proof}
Both directions can be proved straighforwardly by induction from Lemma \ref{lem::algebraic_RGE} or by explicitly writing out the exponential form $\phi = \exp_*(L\sigma)$ using the series expansion of $\exp$.
\end{proof}

We can observe a few facts directly from the explicit form of the tree Feynman rules.  First $\phi(t)$ has $0$ constant term and degree at most $|t|$ in $L$.  The leading term of $\phi(t)$ is $(\sigma(\bullet)L)^{|t|}/t!$.  The linear term of $\phi(t)$ is $\sigma(t)L$.  If $\sigma(\bullet)=1$ and $\sigma$ is $0$ on all other trees then $\phi(t) = L^{|t|}/t!$ which are known as the \textbf{\textit{tree factorial Feynman rules}}.

\bigskip

The reader may find it valuable to understand the connection between these structures and quantum field theory, so we will briefly outline this here.  

Feynman diagrams in quantum field theory give expansions of amplitudes and other physical quantities of interest.  The diagrams are graphs with edges representing particles and vertices representing interactions.  Each graph is associated with an integral via rules which say how to build the integrand out of factors for the edges and the vertices.  These are the Feynman rules.  The Feynman diagrams we are most interested in give divergent integrals.  Renormalization is the process used to fix this and obtain finite quantities.  Part of the subtlety of renormalization is that subdiagrams of a Feynman diagram may already diverge.  One way to do renormalization involves subtracting off divergent subdiagrams in a particular way that was recognized by Kreimer in 1997 as being given by the antipode of a Hopf algebra \cite{Kreimer1997}.

Since what matters most here is the structure of divergent subgraphs inside larger graphs, it is a useful abstraction to only remember this structure rather than the graphs themselves.  We use the Connes-Kreimer Hopf algebra to do this.  In the case that the structure of divergent subgraphs is tree-like then the graph corresponds simply to that tree.  In the case that there are divergent subgraphs that overlap without one being entirely within the other, then the graph corresponds to a sum of trees giving the different ways of picking a maximal tree-like set of divergent subgraphs.  Renormalization is encoded by the antipode of the Connes-Kreimer Hopf algebra \cite{ConnesKreimer}.

The $B_+$ operator on trees corresponds to insertion of subgraphs into another graph.
The Feynman rules on the Feynman diagrams then become maps from trees to some appropriate algebra.  The Feynman rules should be compatible with the Hopf algebra, giving Definition \ref{def::Feynman_rules}.  The $L$ of the target algebra for the Feynman rules corresponds to the logarithm of the momentum running through the Feynman diagram, or the logarithm of the energy scale.

\subsection{Green functions and Dyson-Schwinger equations}

We are interested not so much in single trees as in sequences of trees, or sequences of linear combinations of trees.  Many of the most important examples from the physics perspective are given by functional equations using $B_+$.

\begin{defi}
\label{def::combinatorial_DSE}
A \textbf{\emph{combinatorial Dyson-Schwinger equation}} in $H_{CK}$ is an equation of the form
\[
X(x) = xB_+(f(X(x))
\]
where $f$ is a formal power series with constant term equal to $1$. 
This equation admits a unique solution $X(x)\in H_{CK}[x]$ defined recursively. See Proposition 2 of \cite{FoissyDyson}.
\end{defi}

Writing the same equation without the indeterminate $x$, by the same argument, we get a unique solution $X$ in the graded completion of $H_{CK}$.  All $x$ is doing is keeping track of the graded pieces and it is a matter of taste whether to include it or not.
 Some authors, including one of us in other work, tends to write combinatorial Dyson-Schwinger equations in the slightly different form
\[
Y(x) = 1+xB_+(g(Y(x))) 
\]
but after the substitution $X(x) = Y(x)-1$ and $f(z) = g(z+1)$ the only difference is in whether or not the solution includes a constant term.

A past result of one of us is a characterization of when such Dyson-Schwinger equations have solutions $X(x)$ for which the algebra generated by the coefficients of $X(x)$ is a Hopf subalgebra of $H_{CK}$.  Specifically:

\begin{theo}[Theorem 4 of \cite{FoissyDyson}]
\label{thm::Foissy_characterization_DSE}
Let $f \in \K[[x]]$ such that $f(0) = 1$, and let $A_f$ denote the algebra generated by the coefficients of the unique solutions $X(x)$ to the combinatorial Dyson-Schwinger equation  $X(x) = xB_+(f(X(x))$. Then the following are equivalent:
\begin{enumerate}
    \item $A_f$ is a Hopf subalgebra of $H_{CK}$
    \item There exists $(\alpha, \beta) \in \K^2$ such that $(1 - \alpha\beta x)f'(x) = \alpha f(x)$ 
    \item There exists $(\alpha, \beta) \in \K^2$ such that 
    \begin{enumerate}
        \item $f(x) = 1$ if $\alpha = 0$
        \item $f(x) =e^{\alpha x}$ if $\beta = 0$
        \item $f(x) = (1 - \alpha \beta x)^{-\frac{1}{\beta}}$ if $\alpha \beta \neq 0$.
    \end{enumerate}
\end{enumerate}
\end{theo}

\begin{example}\label{eg binary}
For example, choosing $f(x) = (1+x)^2$ in the theorem is equivalent to choosing $\alpha = 2, \beta = -\frac{1}{2}$.  Writing the formal power series expansion of $X(x) = \sum_{i\geq 0}{x_i}x^i$ then the sequence of $x_i$s gives the sequence of \textbf{\textit{binary rooted trees}}, namely the sequence of linear combinations of trees where each term is a binary tree having coefficient the number of ways to assign left and right children for every vertex. This sequence begins:
\begin{align*}
    x_0 &= \One\\
    x_1 &= \tun \\
    x_2 &= 2\tdeux \\
    x_3 &= 4\ttroisdeux + \ttroisun   \\
    x_4 &= 8\tquatrecinq +  4\tquatredeux + 2\tquatrequatre 
\end{align*}
\end{example}

\begin{example}\label{eg plane}
Another important example is $f(x) = 1/(1-x)$, or equivalently $\alpha=\beta=1$.  In this case, again writing $X(x)=\sum_{i\geq 0} x_ix^i$ we that the sequence of $x_i$ gives all rooted trees weighted by their number of plane embeddings.
\begin{align*}
    x_0 &= \One\\
    x_1 &= \tun \\
    x_2 &= \tdeux \\
    x_3 &= \ttroisdeux + \ttroisun   \\
    x_4 &= \tquatrecinq +  2\tquatredeux + \tquatrequatre + \tquatreun 
\end{align*}
\end{example}

\begin{example}\label{eg ladder 3}
Note that the solution with $\alpha = 1, \beta = -1$ is combinatorially special, and is equivalent to setting $f(x) = 1+x$ in Theorem \ref{thm::Foissy_characterization_DSE}. This case gives the ladder trees of Example \ref{eg ladder 2}.
\end{example}

We get from the combinatorial Dyson-Schwinger equation and its solution to the \textbf{\emph{physical (or analytic) Dyson-Schwinger equation}} by applying the Feynman rules which we notate $\phi$.  For the solution $X(x)$ to a combinatorial Dyson-Schwinger equation let
\[
G(x,L) = \phi(X(x))
\]
be the corresponding \textbf{\emph{Green function}}, or the solution to the corresponding physical Dyson-Schwinger equation.  For us, the Green function is a formal power series in $x$ and $L$.

\medskip

We will, as before, conclude this subsection with some additional comments on the connection to the quantum field theory, which the uninterested reader can again skip.

Dyson-Schwinger equations in quantum field theory are the quantum analogues of the equations of motion.  When expanded in Feynman diagrams, they can be written as recurrence equations for graphs based on insertion.  This diagrammatic form is what our combinatorial Dyson-Schwinger equations on trees correspond to via the correspondence between Feynman diagrams in terms of their insertion structure and rooted trees, as described in the previous subsection.

The diagrammatic solution to a physically relevant Dyson-Schwinger equation should give a Hopf subalgebra since renormalization should be well defined when restricted to the solution.  The Dyson-Schwinger equations which come up in quantum field theory are all of the form in Theorem \ref{thm::Foissy_characterization_DSE}, confirming this.  This also motivates why Hopf subalgebras are physically interesting, particularly ones with one generator in each degree.  Note that the Connes-Moscovici Hopf subalgebra of Example \ref{ex::Connes_Moscovici} is also a Hopf subalgebra with one generator in each degree but is not the solution to a Dyson-Schwinger equation.

As usual, the Feynman diagrams in quantum field theory are really just shorthand for their Feynman integrals.  Applying the Feynman rules to the Dyson-Schwinger equation at the level of Feynman diagrams we can use the way that $B_+$ interacts with the Feynman rules, see for instance section 3.9 of \cite{Lskript}, in order to replace the insertion operator with an integral operator.  The physical Dyson-Schwinger equations are, then, integral equations for the Green functions and this is how they can be found in perturbative quantum field theory sources.  The solutions are the Green functions.

Most of the time we would have not one Dyson-Schwinger equation but a coupled system of equations, one for each propagator and vertex in the theory, and potentially for higher $n$-point functions as well.  The system case can also be interesting for trees \cite{Fsys}, but we will not consider it in this paper.

\subsection{The renormalization group equation and the $\beta$-function}

\label{subsect::RGE_and_beta_function}

Even more important in quantum field theory than Dyson-Schwinger equations is the renormalization group equation. So far we have treated our Green functions $G(x,L)$ as formal series in $x$ and $L$, however, physically $x$ is the coupling, representing the strength of the particle interactions, while $L$ is the log energy.  Naively one might think that $x$ then should be a constant (and hope it is small), but one consequence of renormalization is that $x$ changes with the energy level.  This is known as the \textbf{\textit{running}} of the coupling.

The message for the moment is simply that change in $x$ and change in $L$ are not independent.  The renormalization group equation captures how change in $x$ and change in $L$ affect $G(x,L)$.  Specifically
\begin{equation}
    \label{eq::renormalization_group_equation_analytic}
    \bigg( \frac{\partial}{\partial L} + \beta(x)\frac{\partial}{\partial x} - \gamma(x) \bigg)G(x,L) = 0
\end{equation}

For us $\beta(x)$ and $\gamma(x)$ are simply formal series determined by the physics, though it can be helpful to keep in mind that $\beta$ is physically encoding the flow of the coupling depending on the energy scale.  Some sources pull out a factor of $x$ writing $x\beta(x)$ where we write $\beta(x)$.  Note that if $\beta(x)$ is identically $0$ then the form of the renormalization group equation simplifies substantially and it can be straightforwardly solved by $G(x,L) = \exp(L\gamma(x))$.  This is known as a \textbf{\textit{pure scaling solution}} and is the case physically when the coupling does not run.  The main message here is only that the case when $\beta(x)$ is identically $0$ is a special case recognized by physics.

\medskip

The goal of the present section is to re-interpret the renormalization group equation as a certain linearity condition on the sequence of trees underlying $G(x,L)$.  This reinterpretation comes from a visit of one of the authors to Spencer Bloch and Dirk Kreimer in Chicago in 2014 and owes a particularly large debt to an unpublished note by Bloch during that time \cite{Bnote}.

We begin with $\displaystyle X(x) = 1+\sum_{n=1}^\infty x^nt_n$ where each $t_n\in H_{CK}$ is homogeneous of weight $n$.  Note that we do not assume that $X(x)$ comes from a combinatorial Dyson-Schwinger equation, though such examples are of particular note.  We only assume that the algebra generated by the $t_n$ is a Hopf subalgebra and that $\displaystyle G(x,L) = \phi(X(x)) = 1+\sum_{n=1}^\infty x^n\phi(t_n)$ satisfies the renormalization group equation \eqref{eq::renormalization_group_equation_analytic}.

To make the later indexing more convenient we will use the following notation for the expansions of $\beta(x)$ and $\gamma(x)$
\[
    \beta(x) = \sum_{n = 1}^{\infty}(-\beta_n)x^{n+1} \qquad  \gamma(x) = \sum_{n = 0}^{\infty}\gamma_nx^n.
\]
We will also use the notation $Q_n(L)$ for the polynomial in $L$ obtained by applying $\phi$ to $t_n$.

To begin the argument, we will not even assume that $G(x,L)$ satisfies the renormalization group equation and see how much information we can obtain simply with the subHopf property along with the fact that the Feynman rules are a Hopf algebra morphism.

Write
\begin{equation}\label{eq def of tau}
\cop(t_n) =  \sum_{i = 0}^{n}\tau_{n, n-i} \tensor t_i
\end{equation}
observing that $t_0 = \tau_{n, 0} = 1$ and $\tau_{n,n} = t_n$ and the $\tau_{n,n-i}$ are polynomials in the $t_j$ for $j<n$.  Then Lemma \ref{lem::algebraic_RGE} gives
\begin{align*}
  Q_n(L_1 + L_2) &= m \circ(\phi_{L_1} \tensor \phi_{L_2}) \circ \cop(t_n)\\
  & =  \sum_{i = 0}^{n}\phi(\tau_{n, n-i})(L_1)Q_i(L_2)
\end{align*}
where $\phi(\tau_{n, n-i})(L_1)$ means substituting $L$ with $L_1$ in $\phi(\tau_{n, n-i})\in \K[L]$.
Consequently
\[
\frac{\partial}{\partial L_2}Q_n(L_2) = \frac{\partial}{\partial L_2}Q_n(L_1 + L_2)|_{L_1 = 0} = \sum_{i = 0}^{n}\frac{\partial}{\partial L_1}\phi(\tau_{n, n-i})(L_1)|_{L_1 = 0}Q_i(L_2).
\]
Note that $\dfrac{\partial}{\partial L_1}\phi(\tau_{n, n-i})(L_1)|_{L_1 = 0} \in \K$, so to emphasize that it is a constant define
\[
c_{n,i} := \frac{\partial}{\partial L_1}\phi(\tau_{n, n-i})(L_1)|_{L_1 = 0}
\]
and substituting $L$ for $L_2$ and summing over $n$ we have
\begin{equation}
    \label{eq::differentiated_greens_function_2}
    \frac{\partial }{\partial L}G(x,L) = \sum_{n = 1}^{\infty}\bigg( \sum_{i = 0}^{n}c_{n,i}Q_i(L)\bigg)x^n 
\end{equation}

Next, from our expansion of $\gamma(x)$ and $\beta(x)$ and the definition of $Q_i(L)$ we calculate
\begin{equation}\label{eq other side}
   \left(\gamma(x) - \beta(x)\frac{\partial}{\partial x}\right)G(x,L)
   =  \sum_{n = 0}^{\infty}\bigg(\sum_{i = 0}^{n}\gamma_{n-i}Q_i(L)\bigg)x^n    +\sum_{n = 1}^{\infty} \bigg(\sum_{i = 1}^{n}i\beta_{n - i}Q_i(L) \bigg)x^n
\end{equation}

Now adding the assumption that $G(x,L)$ satisfies the renormalization group equation we see that the left hand sides of \eqref{eq::differentiated_greens_function_2} and \eqref{eq other side} are equal, so comparing coefficients of $x^n$ on the right sides of each equation we obtain
\begin{equation}\label{eq extract coeffs in x}
    \sum_{i = 0}^{n}c_{n,i}Q_i(L)) =  \sum_{i = 0}^{n}\gamma_{n-i}Q_i(L)   +\sum_{i = 1}^{n}i\beta_{n - i}Q_i(L)
\end{equation}
for $n\geq 1$.
    {}From Lemma \ref{lem::explicit_feynman_rules} we see that provided $\sigma(\bullet)$ is non-zero, $Q_i(L)$ has degree $i$ in $L$ and so the $Q_i(L)$ are linearly independent.  Therefore
\[
\gamma_{n} = c_{n,0} \hspace{1cm} \text{for all $n \geq 1$}
\] and for $n\geq 1$ and $1\leq i\leq n$
\begin{equation}
    \label{eq::c_nk_relationship}
    c_{n, i} = c_{n-i,0} + i\beta_{n-i}
\end{equation}
This latter equation tells us that for any fixed value $k\geq 1$ for $n-i$, the resulting sequence $c_{k+i, i}$ is an arithmetic progression, or equivalently is linear as a function of $i$.  So if $G(x,L)$ satisfies a renormalization group equation then the sequences $c_{k+i,i}$ are linear in $i$.

What about the converse?  Suppose we have a $G(x,L)$ for which the algebra generated by the $t_i$ is Hopf and define the $c_{n,k}$ as above.  If the sequences $c_{k+i, i}$ are each linear in $i$ then \eqref{eq::c_nk_relationship} is satisfied for $n\geq 1$ and suitable values of $\beta_{n-i}$, except possibly for the $i=0$ terms. Hence \eqref{eq extract coeffs in x} is satisfied except possibly for the $i=0$ terms.  Note that $n=0$ implies $i=0$ so any discrepancy when $n=0$ is accounted for among the $i=0$ terms.  The $i=0$ terms give a series in $x$.  Since the two sides of \eqref{eq extract coeffs in x} are the coefficients in the $Q_i$ of the right hand sides of \eqref{eq::differentiated_greens_function_2} and \eqref{eq other side} this implies that the left hand sides of those two equations are also equal up to possibly a series in $x$.  Therefore $G(x,L)$ satisfies a generalization of the renormalization group equation of the form
\begin{equation}\label{eq gen lin rge}
\bigg( \frac{\partial}{\partial L} + \beta(x)\frac{\partial}{\partial x} - \gamma(x) \bigg)G(x,L) = \gamma_0(x).
\end{equation}
This discrepancy from the usual renormalization group equation comes from the fact that we discarded the constant terms in building our recurrences thus leading to the need for $\gamma_0$. We will allow $\gamma_0$ in our renormalization group equations in what follows, though the homogeneous case, when $\gamma_0=0$ is of the most physical interest.  The role of $\gamma_0$ relates to the difference between ladders and chains in quantum field theory, as we'll discuss in the examples below.  First there is one more step to obtain the purely combinatorial linearity condition mentioned earlier.    

It remains to interpret this linearity purely on the level of the trees.  To this end it is handy to visualize these sequences by writing the $c_{n,i}$ in a triangle.
\setlength{\tabcolsep}{2pt}
\begin{center}
\begin{tabular}{ccccccccccccccc}
 & & & & & & & $c_{1,0}$ & & & & & & &  \\ 
 & & & & & & $c_{2,1}$ & & $c_{2,0}$ & & & & & & \\  
 & & & & & $c_{3,2}$ & & $c_{3,1}$ & & $c_{3,0}$ & & & & & \\
 & & & & $c_{4,3}$ & & $c_{4,2}$&  &$c_{4,1}$ &  &$c_{4,0}$ & & & & \\
 & & &$c_{5,4}$ &  & $c_{5,3}$ & & $c_{5,2}$ & & $c_{5,1}$ & & $c_{5,0}$ & & & \\
  & &  & &  & &  & $\vdots$ &  & &  & &  & & 
\end{tabular}
\end{center} 
where using the relationship (\ref{eq::c_nk_relationship}) we obtain:
\setlength{\tabcolsep}{2pt}
\begin{center}
\begin{tabular}{ccccccccccccccc}
 & & & & & & & $c_{1,0}$ & & & & & & &  \\ 
 & & & & & & $\beta_1 + c_{1,0}$ & & $c_{2,0}$ & & & & & & \\  
 & & & & & $2\beta_1 + c_{1,0}$ & & $\beta_2 + c_{2,0}$ & & $c_{3,0}$ & & & & & \\
 & & & & $3\beta_1 + c_{1,0}$ & & $2\beta_2 + c_{2,0}$&  &$\beta_3 + c_{3,0}$ &  &$c_{4,0}$ & & & & \\
 & & &$4\beta_1 + c_{1,0}$ &  & $3\beta_2 + c_{2,0}$ & & $2\beta_3 + c_{3,0}$ & & $\beta_4 + c_{4,0}$ & & & $c_{5,0}$  & & \\
  & &  & &  & &  & $\vdots$ &  & &  & &  & & 
\end{tabular}
\end{center}
and where $c_{n,i} = 0$ for all pairs of $n$ and $i$ where $i \geq n$.  Note the linear sequences in the leftward diagonals.

The following proposition gives these coefficients combinatorial meaning.
\begin{prop}\label{prop and def of lambda}
Fix arbitrary Feynman rules $\phi$ and let $\sigma$ be the corresponding infinitesimal character from Lemma \ref{lem::explicit_feynman_rules}. For $i+j = n$, let $\lam_{i, j}$ denote the coefficient of $t_{j} \tensor t_{i}$ in $\cop(t_n)$. Then $c_{n,i} = \lam_{n-i,i}\sigma(t_{n-i})$.
\end{prop}

\begin{proof}
Using notation from the previous discussion, $c_{n,i}$ is defined as:
\[
c_{n,i} := \frac{\partial}{\partial L_1}\phi(\tau_{n, n-i})(L_1)|_{L_1 = 0}
\]
Substituting in our expression for $\phi$ in terms of $\sigma$ (equation \eqref{eq::explicit_tree_feynman_rules}):

\[
c_{n,i} = \frac{\partial}{\partial L_1} \sum_{S \subseteq E(\tau_{n, n-i})}\bigg(\prod_{t\in (\tau_{n, n-i} \setminus S)}\sigma(t)\bigg) \frac{L_1^{|\tau_{n, n-i} \slash (\tau_{n, n-i} \setminus S)|}}{(\tau_{n, n-i} \slash (\tau_{n, n-i} \setminus S))!}|_{L_1 = 0}
\]
Hence each $c_{n,i}$ is obtained from a polynomial in $L_1$. In particular, after performing the indicated derivative and setting $L_1 = 0$, we are left with only the coefficient of the linear term of this polynomial. As discussed following Lemma \ref{lem::explicit_feynman_rules}, for $t$ a tree we have that the linear term of $\phi(t)$ is simply $L\sigma(t)$. Moreover, since $\phi$ is by definition an algebra homomorphism, the lowest degree term of $\phi(F)$ for any forest $F$ is the number of trees in $F$. In particular, $\phi(F)$ has a linear term and contributes to $c_{n,i}$ only if $F$ is a tree. Recalling that $\tau_{n,n-i}$ is a polynomial in the $t_i$'s, we see that
$$\tau_{n,n-i} = \lambda_{n-i, i}t_{n-i} + \text{products of $t_{j}$'s}$$
for some $\lambda_{n-i,i} \in \K$. Since every term other than $\lambda_{n-i, i}t_{n-i}$ is a forest the result follows.
\end{proof}

We can rewrite the array of the $c_{n,i}$, setting $c_{n,i} = \lam_{n, n-i}\sigma(t_{n-i})$ once again for arbitrary infinitesimal character $\sigma: H_{CK} \rightarrow \K$. This leads to:
\begin{figure}[H]
    \centering
      \begin{tabular}{ccccccccccccccc}
    & & & & & & & $\lambda_{1,1}\sigma(t_1)$ & & & & & & &  \\ 
    & & & & & & $\lambda_{2,1}\sigma(t_1)$ & & $\lambda_{1,2}\sigma(t_2)$ & & & & & & \\ 
    & & & & & $\lambda_{3,1}\sigma(t_1)$ & & $\lambda_{2,2}\sigma(t_2)$ & & $\lambda_{1,3}\sigma(t_3)$ & & & & & \\
    & & & & $\lambda_{4,1}\sigma(t_1)$ & & $\lambda_{3,2}\sigma(t_2)$ & & $\lambda_{2,3}\sigma(t_3)$ & & $\lambda_{1,4}\sigma(t_4)$ & & & & \\ 
    & & & $\lambda_{5,1}\sigma(t_1)$&  & $\lambda_{4,2}\sigma(t_2)$&  & $\lambda_{3,3}\sigma(t_3)$&  & $\lambda_{2,4}\sigma(t_4)$&  &$\lambda_{1,5}\sigma(t_5)$ & & & \\ 
    & & & & & & & $\vdots$&
\end{tabular}
    \caption{An array of coefficients.}
    \label{fig::array_of_coefficients}
\end{figure}

\noindent So the condition on the leftward diagonals giving linear sequences is now a statement that the $\lambda_{n, i}$ give linear sequences in $n$ or possibly that some $\sigma(t_i)$ are zero.
If $G(x,L)$ satisfies the renormalization group equation for suitably generic $\phi$, then all the $\lambda_{n,i}$ will be first order sequences.

The converse is slightly more subtle.  In moving to the $\lambda_{n,i}$ we have gotten rid of the outer diagonals of the triangle, where at least one index is $0$.  However, this is exactly accounted for by the $\gamma_0$ in our generalized renormalization group equation, and so if all the leftward diagonal sequences are linear then the associated $G(x,L)$ satisfies a generalized renormalization group equation of the form in \eqref{eq gen lin rge}.

\begin{example}\label{eg ladder 4}
Consider again the ladders from Example~\ref{eg ladder 3}.  Their Green function, with general tree Feynman rules of Lemma~\ref{lem::explicit_feynman_rules}, is
\begin{align*}
G(x,L) & = \sum_{n\geq 0}\phi(\ell_n)x^n \\
& = \sum_{n\geq 0}x^n\sum_{S \subseteq E(\ell_n)}\bigg(\prod_{t\in (\ell_n \setminus S)}\sigma(t)\bigg) \frac{L^{|\ell_n \slash (\ell_n \setminus S)|}}{(\ell_n \slash (\ell_n \setminus S))!} \\
& = \sum_{\substack{n\geq 0\\k\geq 0}}\sum_{\lambda_1+\cdots+ \lambda_k = n} \bigg(\prod_{i=1}^k\sigma(\ell_{\lambda_i})x^{\lambda_i}\bigg) \frac{L^{k}}{k!} \\
& = \exp\left(L\sum_{i\geq 1}\sigma(\ell_i)x^{i}\right)
\end{align*}
where the third equality comes from the observation that cutting a ladder tree at a subset of its edges breaks the ladder into smaller ladders whose sizes form a composition of the size of the original ladder and this correspondence between cuts and integer compositions is bijective.

We can see directly that this satisfies the original renormalization group equation in the special case $\beta(x)=0$.  The triangle of coefficients $\lambda_{i,j}$ for the ladders is
\[
\begin{tabular}{ccccccccccccccc}
    & & & & & & & 1 & & & & & & &  \\ 
    & & & & & & 1 & & 1 & & & & & & \\ 
    & & & & & 1 & & 1 & & 1 & & & & & \\
    & & & & 1 & & 1 & & 1 & & 1 & & & & \\ 
    & & & 1&  & 1&  & 1&  & 1&  &1 & & & \\ 
    & & & & & & & \vdots&
\end{tabular}
\]
We will look at this example more in depth in Section~\ref{subsect::strong_0th_order_seq}.  Physically ladders correspond to \textbf{\emph{rainbow approximations}} in quantum field theory.  For example the graphs in Figure~\ref{fig rainbow} give a rainbow approximation to the fermion propagator in Yukawa theory.  These are well known to be a $\beta(x)=0$ case.
\begin{figure}
    \centering
    \includegraphics{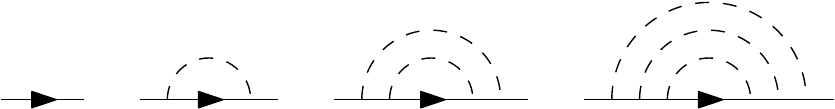}
    \caption{Some Feynman graphs contributing to a rainbow approximation.}
    \label{fig rainbow}
\end{figure}
\end{example}
 
\begin{example}\label{eg corolla 1}
 Another type of approximation that appears in quantum field theory are \textbf{\emph{chain approximations}} or \textbf{\emph{renormalon chains}}.  For example the graphs in Figure~\ref{fig chain} give a chain approximation to the 1PI\footnote{1PI stands for one particle irreducible and is the quantum field theory term for what a graph theorist would call \textbf{\emph{bridgeless}}.} fermion propagator in Yukawa theory. On the tree side, the insertion structure gives rooted trees where all the children of the root are leaves.  These rooted trees are called \textbf{\emph{corollas}}.  
 \begin{align*}
     c_1 & = \tun \\
     c_2 & = \tdeux \\
     c_3 & = \ttroisun \\
     c_4 & = \tquatreun \\
     c_5 & = \tcinqun
 \end{align*}
 The algebra $\K[c_1, c_2, \ldots]$ is subHopf because $\Delta(c_n) = c_n\otimes 1 + 1\otimes c_n + \sum_{i = 1}^{n-1} \binom{n-1}{i} c_1^i\otimes c_{n-i}$.
 The triangle of coefficients $\lambda_{i,j}$ for the corollas is 
 \[
\begin{tabular}{ccccccccccccccc}
    & & & & & & & 1 & & & & & & &  \\ 
    & & & & & & 2 & & 0 & & & & & & \\ 
    & & & & & 3 & & 0 & & 0 & & & & & \\
    & & & & 4 & & 0 & & 0 & & 0 & & & & \\ 
    & & & 5&  & 0&  & 0&  & 0&  &0 & & & \\ 
    & & & & & & & \vdots&
\end{tabular}
\]
The leftmost diagonal sequence here is first order, so we expect a renormalization group equation.  However, this is a case where the initial conditions do not match up with the general recurrence and so we obtain a non-zero $\gamma_0(x)$.  The appearance of a non-zero $\gamma_0(x)$ indicates that chain approximations are not physically as well behaved as the ladders, see Section 7 of \cite{Bms}.

Another interesting observation about corollas is that because of all the zeros in the triangle of coefficients, we can scale the generators to get leftmost diagonal sequences of other orders without causing trouble with other coefficients.  We will return to this idea in subsequent sections.

\begin{figure}
    \centering
    \includegraphics{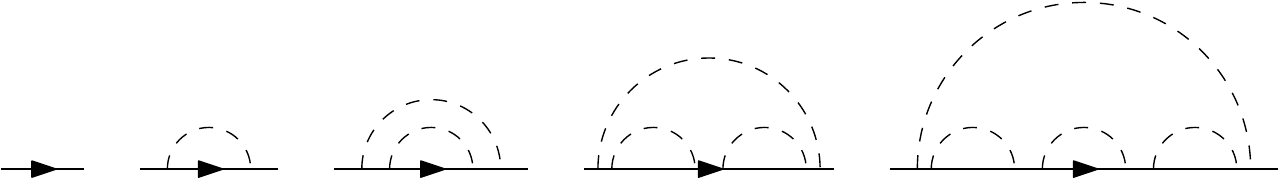}
    \caption{Some Feynman graphs contributing to a chain approximation.}
    \label{fig chain}
\end{figure}
 
\end{example}

\section{Higher-order renormalization group equations}
\label{sect::higher_order_RGEs}

Unfortunately, not every sequence of nonzero linear combinations of trees generating a Hopf subalgebra of $H_{CK}$ satisfies a renormalization group equation. This is the essence of the following lemma:

\begin{lemma}
\label{lem::connes_mosc_doesnt_satisfy_RGE}
Let $s = (t_n)_{n\geq 1}$ be the standard sequence of generators for the Connes-Moscovici Hopf subalgebra of $H_{CK}$ (see Example \ref{ex::Connes_Moscovici}), and let $X_s$ be the corresponding series: $\displaystyle X_{s} = 1 + \sum_{n = 1}^{\infty}t_nx^n$. Then the Green function $G(x,L) = \phi(X)$ does not satisfy a renormalization group equation for any (nonzero) choice of Feynman rules $\phi$.
\end{lemma}

We remark that the Feynman rules $\phi(t_n) = 0$ for all $n \geq 1$ always lead to a Green function which will satisfy a renormalization group equation. As these rules are not very useful, however, we will not consider them in any future deliberations (unless otherwise mentioned).

The key to the proof of Lemma~\ref{lem::connes_mosc_doesnt_satisfy_RGE} is an explicit formula for the $\lam_{i,j}$ for the Connes-Moscovici Hopf algebra.  Before we prove the formula (and from there the lemma) it is useful to gather some intuition by working out initial examples.  Let us compute the left-most diagonal of the array of $\lam_{i,j}$ for the Connes-Moscovici Hopf algebra---these will be the coefficients of $\tun \tensor t_{n-1}$ in $\cop(t_n)$. We get that:
\begin{align*}
    \cop(t_2) &= \ldots +\tun \tensor \tun + \ldots\\
    \cop(t_3) &= \ldots +3\tun \tensor \tdeux + \ldots \\
    \cop(t_4) &= \ldots +6\tun \tensor \bigg(\ttroisdeux + \ttroisun \bigg) + \ldots \\
    \cop(t_5) &= \ldots + 10\tun \tensor \bigg( \tquatrecinq + \tquatrequatre + 3\tquatredeux + \tquatreun\bigg)\\
    &\hspace{2cm} \vdots
\end{align*}
and a pattern has emerged. Indeed, it will not be hard to show that the coefficients we seek form the sequence $(\lam_{i,1})_{i \geq 1} = \left(\binom{i+1}{2}\right)_{i \geq 1}$, hence as a polynomial in $i$ we have $\lam_{i,1} = \frac{1}{2}i^2 + \frac{1}{2}i$. Note that this is a polynomial in $i$ of order $2$. More generally, we will see that $\lam_{i,j} = \binom{i+j}{j+1}$.

Now we proceed to the proof of Lemma~\ref{lem::connes_mosc_doesnt_satisfy_RGE}.
\begin{proof}
In light of the discussion preceding the statement of the Lemma, we only need to check that---for fixed and arbitrary $j$---$\lam_{i,j} \neq ai + b$ for any $a,b \in \K$.

To do this we prove that $\lam_{i,j} = \binom{i+j}{j+1}$ by direct counting (see also Example~\ref{eg::connes_moscovici_example}).  Labelling each vertex of a rooted tree by the step in which it is added by the natural growth operator in the construction of the Connes-Moscovici Hopf algebra, we see that we obtain trees with increasing labellings, that is labellings with consecutive integers starting at 1 where every vertex has a label larger than its parent.  Furthermore, all increasing labellings are obtained exactly once.  (This justifies one of the standard alternate formulations of the Connes-Moscovici Hopf algebra is as the Hopf algebra of increasing labellings.)  From this perspective calculating $\lam_{i,j}$ is calculating the number of ways an increasingly labelled tree $t_1$ of size $j$ can be grafted onto an increasingly labelled tree $t_2$ of size $i$ where the labels are shuffled together while maintaining the increasing labelling property; that is, after the grafting we obtain an increasingly labelled tree where the labelling restricted to either $t_1$ or $t_2$ and normalized gives the initial labelling of that tree. 

The number of ways to do this depends on the label of the vertex onto which we graft.  Suppose we graft onto vertex $v$ of $t_2$ which has label $a$.  Then to obtain a labelling of the grafted tree we have the choice of any shuffle of the labels of $t_1$ with the labels larger than $a$ in $t_2$.  There are $\binom{j+i-a}{j}$ ways to do this.  Now summing over the vertex onto which we graft we get $\sum_{a=1}^{i} \binom{j+i-a}{j}$ which equals $\binom{i+j}{j+1}$ by the hockey-stick identity for binomial coefficients.

Therefore, we have that (for $j$ fixed arbitrarily) $(\lam_{i,j})_{i \geq 1} = \left(\binom{i+j}{j+1}\right)_{i \geq 1}$, which for all $j$ forms a sequence of $i$ of order at least $2$. Hence there is no choice of (nonzero) Feynman rules $\phi$ that will make the array of $c_{n,i}$'s have linear left diagonals, and consequently $G(x,L) = \phi(X)$ will not satisfy a renormalization group equation for any choice of Feynman rules $\phi$.
\end{proof}

Considering the Connes-Moscovici example further, note that by choosing Feynman rules with $\sigma(t)=0$ for $t\neq \tun$ we are left with only the leftmost diagonal in the array of $\lambda_{i,j}$, that is with the $\lambda_{i,1}$ which give a quadratic sequence.  This naturally leads us to think about analogues of the renormalization group equation which give quadratic or higher order sequences instead of only linear sequences.

\begin{defi}[Generalized Renormalization Group Equations]
\label{def::generalized_renormalization_group_equations}
For a Green function $G(x,L)$, define a {\bf generalized renormalization group equation} by:
\begin{equation}
    \label{eq::generalized_RGE}
    \frac{\partial G}{\partial L} = \overline{\beta}(x, \frac{\partial}{\partial x})G + \gamma_0(x)
\end{equation}
where $\overline{\beta}$ is polynomial in its second argument and $\gamma_0(x)$ is a series in $x$.  If the polynomial $\overline{\beta}$ is of degree $n$ in $\frac{\partial}{\partial x}$, we say that the generalized renormalization group equation is {\bf of order $n$}.  If $\gamma_0(x)=0$ we say that the generalized renormalization group equation is \textbf{homogeneous}.
\end{defi}
We remark that the standard definition of the renormalization group equation can be seen as a generalized renormalization group equation of order $1$ and homogeneous. Indeed, we simply take $\overline{\beta}(x, \frac{\partial}{\partial x}) = \gamma(x) + \beta(x)\frac{\partial}{\partial x}$. We also note that this definition is informed by quantum field theory, as the case of $\beta \equiv 0$ is already a known special case in physics as discussed above. In the language we are introducing with Definition \ref{def::generalized_renormalization_group_equations} these are $0$th-order homogeneous renormalization group equations. The homogeneous case is the most interesting from a physics perspective, but the example of the corollas (Example~\ref{eg corolla 1}) shows that nonhomogeneous examples do also appear in quantum field theory.

We can ask ourselves the same question about equation (\ref{eq::generalized_RGE}) that we did in Section \ref{subsect::RGE_and_beta_function} regarding the usual renormalization group equation: namely, what does it mean for a Green function to satisfy equation (\ref{eq::generalized_RGE}) after translating to combinatorics?  Let us approach the question using the same method as before; we will turn equation (\ref{eq::generalized_RGE}) into a statement about formal power series, and then compare coefficients across the equals sign.

To fix notation again, let $\phi$ be our Feynman rules, $(t_n)_{n \geq 1}$ be a sequence of linear combinations of trees generating a Hopf subalgebra of $H_{CK}$, and let $X = 1 + \sum_{n = 1}^{\infty} t_nx^n$ be the corresponding series. Then $G(x,L) = \phi(X)$ as before. Now $\overline{\beta}$ is allowed to be more general than $\beta$ was before, so we will set up the notation in the following way:
\begin{equation}
    \overline{\beta}(x, \frac{\partial}{\partial x}) := \sum_{j = 0}^{m}\beta^{(j)}(x)\frac{\partial^j}{\partial x^j}
\end{equation}
with each $\beta^{(i)}(x)$ a formal power series in $x$:
\begin{equation}
    \label{eq::def_of_beta_j_power_series}
    \beta^{(j)}(x) := \sum_{k = 1}^{\infty}\beta_{k}^{(j)}x^{k+1}
\end{equation}
with $\beta_{k}^{(j)} \in \K$. 

With this setup, the usual renormalization group equation is recovered by setting $\beta^{(0)}(x) = \gamma(x)$, $\beta^{(1)}(x) = \beta(x)$, and $\beta^{(k)}(x) \equiv 0$ for $k \geq 2$.

\label{gRGE derivation}

Since our Green function $G(x,L)$ looks exactly as it did before, the left-hand side of (\ref{eq::generalized_RGE}) is also precisely the same as shown in \eqref{eq::differentiated_greens_function_2}. We write it again for convenience:
\[
    \frac{\partial }{\partial L}G(x,L) = \sum_{n = 1}^{\infty}\bigg( \sum_{i = 0}^{n}c_{n,i}Q_i(L))\bigg)x^n
\]
So all that remains is to find what the right-hand side of equation (\ref{eq::generalized_RGE}) looks like in terms of formal power series.  We compute:
\begin{align*}
    \overline{\beta}G &= \bigg(\sum_{j = 0}^{m}\beta^{(j)}(x)\frac{\partial^j}{\partial x^j} \bigg)\bigg( \sum_{n = 0}^{\infty}Q_n(L)x^n\bigg)\\
    &= \beta^{(0)}\bigg( \sum_{n = 0}^{\infty}Q_n(L)x^n\bigg) + \beta^{(1)}\frac{\partial}{\partial x}\bigg( \sum_{n = 0}^{\infty}Q_n(L)x^n\bigg) + \ldots + \beta^{(m)}\frac{\partial^m}{\partial x^m}\bigg( \sum_{n = 0}^{\infty}Q_n(L)x^n\bigg) \\ 
    &= \beta^{(0)}\bigg( \sum_{n = 0}^{\infty}Q_n(L)x^n\bigg) + \beta^{(1)}\bigg( \sum_{n = 0}^{\infty}nQ_n(L)x^{n-1}\bigg) \\
    &+ \ldots + \beta^{(m)}\bigg( \sum_{n = 0}^{\infty}n(n-1)\cdots(n - m + 1)Q_n(L)x^{n-m}\bigg) \\ 
    &= \sum_{j = 0}^{m}\sum_{n = 0}^{\infty}\beta^{(j)}n(n-1)\cdots(n - j + 1)Q_n(L)x^{n-j} \\
    &= \sum_{n = 0}^{\infty}\sum_{j = 0}^{m}\beta^{(j)}n(n-1)\cdots(n - j + 1)Q_n(L)x^{n-j} \\
    &= \sum_{n = 0}^{\infty}\sum_{j = 0}^{\text{min}(m, n-1)}\beta^{(j)}\frac{n!}{(n-j)!}Q_n(L)x^{n-j}
\end{align*}
But now the $\beta^{(j)}$ are just power series, with coefficients as defined in (\ref{eq::def_of_beta_j_power_series}), so we may substitute these in. When we do this we get:
\begin{align*}
    \overline{\beta}G
    &= \sum_{n = 0}^{\infty}\sum_{j = 0}^{\text{min}(m, n-1)} \sum_{k = 1}^{\infty}\beta_{k}^{(j)}\frac{n!}{(n-j)!}Q_n(L)x^{n-j + k +1} 
\end{align*}
Now we want the sum to be indexed by powers of $x$, so we make a substitution on the indices: $t = n - j + k + 1$. This yields:
\begin{align*}
    \overline{\beta}G &= \sum_{t = 2}^{\infty}\sum_{j = 0}^{\text{min}(m, n-1)} \sum_{n - j + k +1 = t}\beta_{k}^{(j)}\frac{n!}{(n-j)!}Q_n(L)x^{t}
\end{align*}
Finally we can simplify the indices of the third summation as well. If $n - j + k = t$, and $j$ is already fixed by the second summation, it follows that $n$ and $k$ are partitioning $t + j + 1$ and hence we can rewrite this as:
\begin{align*}
    \overline{\beta}G &= \sum_{t = 1}^{\infty}\sum_{j = 0}^{m} \sum_{n = 0}^{t + j + 1}\beta_{t + j - n - 1}^{(j)}\frac{n!}{(n-j)!}Q_n(L)x^{t} 
\end{align*}

Now up to this point we have only used the assumption that the underlying sequence in question generated a Hopf subalgebra. As in the case of the usual renormalization group equation, we can then add in the assumption that $G(x,L)$ satisfies the generalized renormalization group equation \eqref{eq::generalized_RGE} to create relations among coefficients of the respective power series. Doing this, we obtain a form of the generalized renormalization group equation expressed solely in terms of formal power series:
\begin{equation}
\label{eq::generalized_RGE_expanded}
    \sum_{n = 0}^{\infty}\bigg( \sum_{i = 0}^{n}c_{n,i}Q_i(L))\bigg)x^n  = \sum_{n = 0}^{\infty}\sum_{j = 0}^{m} \sum_{i = 0}^{n + j - 1}\beta_{t + j - i - 1}^{(j)}\frac{i!}{(i-j)!}Q_i(L)x^{n} 
\end{equation}
(Note that we have changed the letters of some of the indices to avoid confusion). This is the generalized version of equation \eqref{eq extract coeffs in x}. By comparing the coefficients of powers of $x$ on each side, we conclude that the following identity holds for all $n \geq 1$:
\begin{equation}
\label{eq::Q_i_generalized_RGE}
    \sum_{i = 0}^{n}c_{n,i}Q_i(L) = \sum_{j = 0}^{m} \sum_{i = 0}^{n + j - 1}\beta_{n + j - i - 1}^{(j)}\frac{i!}{(i-j)!}Q_i(L)
\end{equation}
and using as before that the $Q_i(L)$ are linearly independent, we can compare their coefficients across the equals sign as well to get that:
\begin{align}
\label{eq::c_n_i_generalized_RGE}
    c_{n,i} &= \sum_{j = 0}^{m}\beta_{n + j - i}^{(j)}\frac{i!}{(i-j)!}\\
    &= \beta_{n - i}^{(0)} + i\beta_{n - i + 1}^{(1)} + \ldots + i(i-1)\cdots(i- m + 1)\beta_{n - i + m}^{(m)}
\end{align}
In other words, when $n$ is fixed we have that $c_{n,i}$ is a polynomial in $i$ of degree $m$. Hence we have found a necessary condition for which $G(x,L)$ satisfy an $m$th-order generalized renormalization group equation, as we desired. 

As before, it remains to consider the converse. Suppose we start with $G(x,L)$ such that the corresponding sequence of $t_i$ generate a Hopf subalgebra with the $c_{n,i}$ polynomial in $i$ of degree $m$. Call these polynomials $f_j(i)$.  For $n \geq 1$, we can recover the coefficients $\beta_n^{(j)}$ by evaluating the polynomial $f_j$ at $i = 0, \ldots , m$ and then changing basis for the vector space of polynomials as in Section 4.3 of \cite{StanleyVolume1}. Hence as in the case of the usual renormalization group equation, equation $\eqref{eq::c_n_i_generalized_RGE}$ is satisfied except possibly in degree $i = 0$. We remark that this correspondence recovering $\overline{\beta}$ is possible exactly because each $\beta_i^{(j)}$ appears as the coefficient in a unique leftward diagonal of the triangle $(c_{i,j})_{i,j \geq 1}$. For generic $m \geq 3$ the triangle begins:

\setlength{\tabcolsep}{-5pt}
\begin{center}
\begin{tabular}{ccccccccccccccc}
 & & & & & & & $\beta_1^{(0)}$ & & & & & & &  \\ 
 & & & & & & $\beta_1^{(0)} + \beta_1^{(1)}$ & & $\hspace{1cm}\beta_2^{(0)}$ & & & & & & \\  
 & & & & & $\beta_1^{(0)} + 2\beta_1^{(1)} + 2\beta_1^{(2)}$ & & $\beta_2^{(0)} + \beta_2^{(1)}$ & & $\hspace{1cm}\beta_3^{(0)}$ & & & & & \\
 & & & & $\beta_1^{(0)} + 3\beta_1^{(1)} + 6\beta_1^{(2)} + 6\beta_1^{(3)}$ & & $\beta_2^{(0)} + 2\beta_2^{(1)} + 2\beta_2^{(2)}$&  &$\hspace{1cm}\beta_3^{(0)} + \beta_3^{(01)}$ &  &$\hspace{2cm}\beta_4^{(0)}$ & & & & \\
  & &  & &  & &  & $\vdots$ &  & &  & &  & & 
\end{tabular}
\end{center}
Now by thinking of each $Q_i(L)$ as a formal variable and inserting the $c_{n,i}$'s we can recover equation \eqref{eq::Q_i_generalized_RGE} except for possible disagreement in the $i = 0$ terms. However when $i = 0$ in equation \eqref{eq::Q_i_generalized_RGE}, the only part of $\overline{\beta}$ that survives is $\beta^{(0)}(x)$. This means that the power series expression of the generalized renormalization group equation \eqref{eq::generalized_RGE_expanded} holds up to a power series in $x$. This is the origin of the inhomogeneity $\gamma_0(x)$ as before. In conclusion, assuming only that the sequence of $t_i$ corresponding to $G(x,L)$ generate a Hopf subalgebra with the $c_{n,i}$ polynomial in $i$ of degree $m$, we have that 
\begin{equation}
    \label{eq::generalized_RGE_inhomogeneous}
    \left( \frac{\partial }{\partial L} - \overline{\beta}(x, \frac{\partial}{\partial x}) \right)G= \gamma_0(x)
\end{equation}
We can summarize the preceding discussion with the following lemma:

\begin{lemma}\label{lem summary equiv}
A Green function $G(x,L)$ satisfies an $m$th order non-homogeneous generalized renormalization group equation if and only if for each $n$, $c_{n,i}$ is a polynomial of degree $m$ in $i$.
\end{lemma}

\section{$\Lambda$-arrays, strong and weak sequences}\label{sec lambda}

As discussed in the introduction, in order to pin down the scaling freedom we will be thinking not of the Hopf subalgebras directly, but in terms of a sequence of generators.  In the previous section we saw that the array of coefficients $\lambda_{i,j}$ is key to reading off the order of renormalization group equation that the associated Green function satisfies.  Now it is time to make these two notions of sequence of generators and doubly-indexed sequences $(\lambda_{i,j})_{i,j \geq 1}$ precise and show that---suitably defined---they are equivalent.  With that equivalence in place we will be ready to define our combinatorial notions of strong and weak $k$th order sequences.

\subsection{Equivalence between sequences and $\Lambda$-arrays}

\begin{defi}
We shall consider sequences 
$(t_n)_{n\geq 1}$ of elements of $H_{CK}$ such that:
\begin{enumerate}
\item For all $n\geq 1$, $t_n$ is a nonzero linear combination of trees of degree $n$.
\item $t_1=\bullet$.
\item The subalgebra $A$ generated by the elements $t_n$, $n\geq 1$, is a Hopf subalgebra of $H_{CK}$.
\end{enumerate}
The set of such sequences is denoted by $\seq$.
\end{defi}

\begin{remark}
We could also consider the set $\seq'$ of sequences $(t_n)_{n\geq 1}$ satisfying points 1 and 3 only.
If $(t_n)_{n\geq 1} \in \seq'$, there exists a nonzero scalar $\alpha$ such that $t_1=\alpha \tun$.
Hence, there exists a bijection:
\[\left\{\begin{array}{rcl}
\K\setminus\{0\}\times \seq&\longrightarrow&\seq'\\
(\alpha,(t_n)_{n\geq 1})&\longrightarrow&(\alpha t_1,t_2,t_3,\ldots).
\end{array}\right.\]
Consequently, it is enough to consider $\seq$.
\end{remark}

It turns out we can fully characterize exactly which $(\lambda_{i,j})_{i,j \geq 1}$ can appear from elements of $\seq$. This is the content of the following:
\begin{theo} \label{theo1}
We denote by $\Lambda$ the set of double sequences $(\lambda_{i,j})_{i,j\geq 1}$ such that:
\begin{enumerate}
\item (Pre-Lie relation). For any $i,j,k\geq 1$:
\begin{equation}
\label{eq::prelie_relation}
 \lambda_{i,j}\lambda_{i+j,k}-\lambda_{j,k}\lambda_{i,j+k}
=\lambda_{i,k}\lambda_{i+k,j}-\lambda_{k,j}\lambda_{i,j+k}.   
\end{equation}

This relation is denoted by $PL(i,j,k)$.
\item (Non degeneracy). For any $n\geq 2$, there exist $i,j \geq 1$ such that $i+j=n$ and $\lambda_{i,j}\neq 0$.
\end{enumerate}
Then there is a bijection between $\seq$ and $\Lambda$ given by the definition of $\lambda_{i,j}$ in Proposition~\ref{prop and def of lambda}.
\end{theo}

To prove the theorem it will be handy to have the following lemma.

\begin{lemma} \label{lemme2}
Let $(\lambda_{i,j})_{i,j\geq 1} \in \Lambda$. We define a pre-Lie structure on the space $V=Vect(X_i,i\geq 1)$ by:
\[X_i\bullet X_j=\lambda_{i,j} X_{i+j}.\]
It is graded, with $X_i$ homogeneous of degree $i$ for any $i$.
Moreover, $(V,\bullet)$ is generated by $X_1$.
\end{lemma}

\begin{proof}
The pre-Lie relation for $(X_i,X_j,X_k)$ is equivalent to point 1 of the theorem statement. Let us denote by $V'$ the pre-Lie subalgebra
of $(V,\bullet)$ generated by $X_1$ and let us prove inductively that $X_k \in V'$.
If $k=1$, this is obvious. Otherwise, there exist $i,j$, such that $i+j=k$ and $\lambda_{i,j}\neq 0$.
By the induction hypothesis, $X_i,X_j \in V'$, so $\displaystyle X_{i+j}=\frac{1}{\lambda_{i,j}} 
X_i\bullet X_j \in V'$.
\end{proof}

\begin{proof} (Theorem \ref{theo1}). 
We aim to define $\theta:\seq\longrightarrow\Lambda$.  Let the $\lambda_{i,j}$ for $\theta(T)$, $T\in\seq$ be as in Proposition~\ref{prop and def of lambda}.  Then we need to show that these $\lambda_{i,j}$ satisfy the conditions above and that $\theta$ is bijective. 

It will be convenient to work in the graded dual.  The graded dual of $H_{CK}$ is $H_{GL}$, the Grossman-Larson Hopf algebra of rooted trees \cite{Grossman,Panaite,Hoffman}.
The space of primitive elements of $H_{GL}$ has for basis the set of rooted trees.
It inherits a product $\bullet$ defined by:
\[t\bullet t'=\sum_{\mbox{\scriptsize $s$ vertex of $t$}}\mbox{grafting of $t'$ on $s$}.\]
This product is (right) pre-Lie, that is to say, for any $x,y,z\in Prim(H_{GL})$:
\[(x\bullet y)\bullet z-x\bullet(y\bullet z)=(x\bullet z)\bullet y-x\bullet(z\bullet y).\]
By \cite{Chapoton}, the pre-Lie algebra $\g=(Prim(H_{GL}),\bullet)$ is freely generated by the  single vertex tree $\tun$.
Note that for any trees $t$, $t'$, we have  $t*t'=tt'+t\bullet t'$ with $*$ as the multiplication in $H_{GL}$. Hence, for any $x,y\in \g$:
\begin{align}
\label{EQ1} x*x'&=xx'+x\bullet x'.
\end{align}

The pairing between $H_{CK}$ and $H_{GL}$ is given by the following: if $f$ and $f'$ are two rooted forests,
\[\langle f,f'\rangle=s_f \delta_{f,f'},\]
where $s_f$ is the number of symmetries (as a rooted graph) of $f$. 

Let $(t_n)_{n\geq 1}$ be an element of $\seq$. We denote by $A$ the subalgebra of $H_{CK}$ generated by the
elements $t_n$, $n\geq 1$. This is a graded Hopf subalgebra of $H_{CK}$, with a basis given
by monomials of $t_n$, as none of the $t_n$ is zero. Let $I=A^\perp\subseteq H_{GL}$.
As $A$ is a graded Hopf subalgebra of $H_{CK}$, $I$ is a graded biideal of $H_{GL}$
and the graded dual $A^*$ of $A$ is identified with $H_{GL}/I$. 
As $H_{GL}$ is connected and cocommutative, it is the enveloping algebra of $\g$ by the Cartier-Quillen-Milnor-Moore's
theorem. Consequently, $A^*$ is the enveloping algebra of $\g/\g\cap I$. 
Furthermore, $\g/\g\cap I$ is a pre-Lie ideal of $\g$ and inherits a pre-Lie product, also denoted by $\bullet$.\\

Let us consider the basis $(e_k)_{k\geq 1}$ of $\g/\g\cap I$ defined by:
\[\langle e_k,t_l\rangle=\delta_{k,l}.\]
In particular, as $t_1=\tun$, $e_1=\tun+I$.
For any $k\geq 1$, $e_k$ is homogeneous of degree $k$. As the pre-Lie product of $\g/\g\cap I$ is homogeneous,
for any $i,j\geq 1$, there exists a scalar $\lambda_{i,j}$ such that:
\[e_i\bullet e_j=\lambda_{i,j} e_{i+j},\]
and tracing through the dualities these $\lambda_{i,j}$ are exactly the $\lambda_{i,j}$ of the map $\theta$.
Writing the pre-Lie relation of $\g/\g\cap I$ for $(e_i,e_j,e_k)$ gives property 1 for $(\lambda_{i,j})_{i,j\geq 1}$.
Let us now prove property 2 for this double sequence. Let us assume that there exists $n\geq 2$
such that for any $i,j$ such that $i+j=n$, $\lambda_{i,j}=0$.
Let us put $\g'=Vect(e_k,k\neq n)$. For any $i,j\geq 1$:
\[e_i\bullet e_j =\lambda_{i,j}e_{i+j}\in \begin{cases}
(0)\mbox{ if }i+j=n,\\
\g' \mbox{ otherwise}.
\end{cases}\]
Hence, $\g'$ is a strict pre-Lie subalgebra of $\g/\g\cap I$ containing $e_1$.
As $\g$ is generated by the single vertex tree $\tun$, it follows that $\g/\g\cap I$ is generated by $e_1$: this is a contradiction.
So property 2 is satisfied. Therefore $\theta$ is a map from $\seq$ to $\Lambda$.

Let us prove that $\theta$ is injective. Let us assume that $\theta((t_n)_{n\geq 1})=\theta((t'_n)_{n\geq 1})
=(\lambda_{i,j})_{i,j\geq 1}$. Let $V$ be the pre-Lie algebra associated to $(\lambda_{i,j})_{i,j\geq 1}$ by Lemma \ref{lemme2}.
Then, with the preceding notations, we obtain graded  pre-Lie algebra isomorphisms:
\begin{align*}
\phi&:\left\{\begin{array}{rcl}
\g/\g \cap I&\longrightarrow&V\\
e_i&\longrightarrow&X_i,
\end{array}\right.&
\phi'&:\left\{\begin{array}{rcl}
\g/\g \cap I'&\longrightarrow&V\\
e'_i&\longrightarrow&X_i.
\end{array}\right.
\end{align*}
Extending this Lie algebra morphism, we obtain graded Hopf algebra isomorphisms:
\begin{align*}
\Phi&:\left\{\begin{array}{rcl}
H_{GL}/I&\longrightarrow&\mathcal{U}(V)\\
e_i&\longrightarrow&X_i,
\end{array}\right.&
\phi'&:\left\{\begin{array}{rcl}
H_{GL}/I'&\longrightarrow&\mathcal{U}(V)\\
e'_i&\longrightarrow&X_i.
\end{array}\right.
\end{align*}
Dually, we obtain Hopf algebra monomorphisms:
\begin{align*}
\Phi^*&:\left\{\begin{array}{rcl}
\mathcal{U}(V)^*&\longrightarrow&H_{CK}\\
X_i^*&\longrightarrow&t_i,
\end{array}\right.&
\Phi'^*&:\left\{\begin{array}{rcl}
\mathcal{U}(V)^*&\longrightarrow&H_{CK}\\
X_i&\longrightarrow&t'_i.
\end{array}\right.
\end{align*}
As $\g$ is freely generated by $\tun$, there exists a unique pre-Lie algebra morphism $\psi$ from $\g$ to $V$
sending $\tun$ to $X_1$. Obviously, $\psi=\phi \circ \pi=\phi'\circ \pi'$, where $\pi$ and $\pi'$
are the canonical surjections on $\g/\g\cap I$ and $\g/\g\cap I'$ respectively. Hence,
$I=I'$ and $\phi=\phi'$. Consequently, $\Phi=\Phi'$, $\Phi^*=\Phi'^*$ and finally $(t_n)_{n\geq 1}=(t'_n)_{n\geq 1}$.
So $\theta$ is injective.\\

Let us prove that $\theta$ is surjective. Let $(\lambda_{i,j})_{i,j\geq 1}\in \Lambda$ and let us consider the
pre-Lie algebra $V$ of Lemma \ref{lemme2}. As $V$ is generated by $X_1$, there exists a unique pre-Lie algebra
morphism $\phi:\g\longrightarrow V$, surjective, such that $\phi(\tun)=X_1$. 
This morphism is extended as a surjective Hopf algebra morphism $\Phi:H_{GL}\longrightarrow \mathcal{U}(V)$.
Dually, we obtain an injective Hopf algebra morphism $\Phi^*:\mathcal{U}(V)^*\longrightarrow H_{CK}$.
If we put $t_n=\Phi^*(X_n^*)$ for all $n\geq 1$, then $t_n$ is a nonzero linear combination of trees 
of degree $n$ and $\Phi^*(\mathcal{U}(V)^*)=\K[t_n,n\geq 1]$ is a Hopf subalgebra of $H_{CK}$.
As $\phi(\tun+I)=X_1$, we have $t_1=\tun$. By construction, $\theta((t_n)_{n\geq 1})=(\lambda_{i,j})_{i,j\geq 1}$. \end{proof}

It is also convenient to describe the inverse bijection explicitly. We shall use the Oudom-Guin construction of the enveloping
algebra of a pre-Lie algebra \cite{OudomGuin}.
Let $(\lambda_{i,j})_{i,j\geq 1}\in \Lambda$ and let $V$ be the pre-Lie algebra of Lemma \ref{lemme2}.
The pre-Lie product $\bullet$ is extended to the symmetric Hopf algebra $S(V)$ in the following way:
\begin{enumerate}
\item For any $x \in S(V)$, $x\bullet 1=x$.
\item For any $x,y,z\in S(V)$, with Sweedler's notation:
\[xy\bullet z=\left(x\bullet z^{(1)}\right)\left(y\bullet z^{(2)}\right).\]
\item For any $a,b_1,\ldots,b_k \in S(V)$, with $k\geq 2$:
\[a\bullet b_1\ldots b_k=(a\bullet b_1\ldots b_{k-1})\bullet b_k-\sum_{i=1}^{k-1}
a\bullet (b_1\ldots (b_i \bullet b_k)\ldots b_{k-1}).\]
\end{enumerate}
We shall consider the element $X_1\bullet X_{i_1}\ldots X_{i_k}$ of $V$.. By homogeneity, there exists
a scalar $\lambda(i_1,\ldots,i_k)$ such that:
\[X_1\bullet X_{i_1}\ldots X_{i_k}=\lambda(i_1,\ldots,i_k)X_{1+i_1+\ldots+i_k}.\]
These coefficients are computed by induction on $k$:
\[\lambda(i_1,\ldots,i_k)=\begin{cases}
1\mbox{ if }k=0,\\
\lambda_{1,i_1}\mbox{ if }k=1,\\
\lambda(i_1,\ldots,i_{k-1})\lambda_{1+i_1+\ldots+i_{k-1},i_k}
\displaystyle -\sum_{j=1}^{k-1}\lambda(i_1,\ldots,i_j+i_k,\ldots,i_{k-1})\lambda_{i_j,i_k}
\mbox{ otherwise}.
\end{cases}\]
By construction, $\lambda(i_1,\ldots,i_k)$ is symmetric in its arguments.

Let us denote by $\phi:\g\longrightarrow V$ the unique pre-Lie algebra morphism sending $\tun$ to $X_1$.
For any tree $t$, we put $\phi(t)=\mu(t) X_{|t|}$ by homogeneity. Then, if $t=B^+(t_1\ldots t_k)$:
\begin{align*}
\phi(t)&=\phi(\tun \bullet t_1\ldots t_k)\\
&=X_1\bullet \phi(t_1)\ldots \phi(t_k)\\
&=\mu(t_1)\ldots \mu(t_k) X_1\bullet X_{|t_1|}\ldots X_{|t_k|}\\
&=\mu(t_1)\ldots \mu(t_k) \lambda(|t_1|,\ldots,|t_k|) X_{|t|}.
\end{align*}
Therefore, $\mu(t)$ is given by:
\begin{align*}
\mu(t)&=\prod_{s \in Vert(t)} \lambda\left(t_1^{(s)},\ldots, \lambda_{k(s)}^{(s)}\right),
\end{align*}
where $k(s)$ is the fertility of the vertex $s$ and $t_1^{(s)},\ldots, t_{k(s)}^{(s)}$ are
the subtrees of $t$ born from $s$.
Denoting by $\Phi:H_{GL}\longrightarrow \mathcal{U}(\g)$ the extension of $\phi$ to $H_{GL}$ 
and by $\Phi^*:\mathcal{U}(g)^*\longrightarrow H_{CK}$ its transpose, then by duality,
putting $(t_n)_{n\geq 1}=\theta^{-1}((\lambda_{i,j})_{i,j\geq 1})$, for all $n\geq 1$:
\[t_n=\Phi^*(X_n^*)=\sum_{|t|=n} \frac{\mu(t)}{s_t}t.\]

\begin{example}
\label{eg::dyson_schwinger_example}
Consider sequences coming from the solution to a Dyson Schwinger equation.  We already know from Theorem~\ref{thm::Foissy_characterization_DSE} that these sequences are in $\seq$, but they are also useful to look at from the $\Lambda$ perspective.  In fact we can check directly that their $\Lambda$-arrays are in $\Lambda$ as follows.

Let $(a,b)\in \K^2$ such that $a+b\neq 0$. For all $i$, we put $\lambda_{i,j}=ai+b$.
For any $i,j,k\geq 1$:
\begin{align*}
\lambda_{i,j}\lambda_{i+j,k}-\lambda_{j,k}\lambda_{i,j+k}
&=(ai+b)(a(i+j)+b)-(a(j-1)+b)(ai+b)\\
&=(ai+b)a(i+1)\\
&=\lambda_{i,k}\lambda_{i+k,j}-\lambda_{k,j}\lambda_{i,j+k}.
\end{align*}
So property 1 is satisfied. For any $n\geq 2$, $\lambda_{1,n-1}=a+b\neq 0$, so property 2 is satisfied. 
An easy induction proves that for any $k\geq 0$:
\[\lambda(i_1,\ldots,i_k)=\prod_{j=1}^k (a(j-k+1)+b).\]
Hence, for any $n\geq 0$:
\[t_n=\sum_{|t|=n} \frac{1}{s_t}\left(\prod_{s\in Vert(t)} \prod_{j=1}^{k(s)} (a(j-k(s)+1)+b)\right) t.\]
For example:
\begin{align*}
t_2&=(a+b) \tdeux,\\
t_3&=(a+b)^2\ttroisdeux+\frac{(a+b)a}{2}\ttroisun,\\
t_4&=(a+b)^3\tquatrecinq+\frac{a(a+b)^2}{2}\tquatrequatre+a(a+b)^2\tquatredeux+\frac{(a+b)a(a-b)}{6}\tquatreun.
\end{align*}
This is the solution of the Dyson-Schwinger equation $X=B^+((1+bX)^{\frac{a+b}{b}})$ if $b\neq 0$
and $X=B^+(e^{aX})$ if $b=0$ of \cite{FoissyDyson}; see Theorem \ref{thm::Foissy_characterization_DSE}.  This includes as special cases the ladders of Example~\ref{eg ladder 4}, the binary rooted trees of Example~\ref{eg binary}, and the plane rooted trees of Example~\ref{eg plane}.
\end{example}

\begin{example}
\label{eg::connes_moscovici_example}
Continuing the Connes-Moscovici example (see Example~\ref{ex::Connes_Moscovici} and Lemma~\ref{lem::connes_mosc_doesnt_satisfy_RGE}), 
for any $i,j\geq 1$, put:
\[\lambda_{i,j}=\binom{i+j}{i-1}.\]
For any $i,j,k\geq 1$:
\begin{align*}
\lambda_{i,j}\lambda_{i+j,k}-\lambda_{j,k}\lambda_{i,j+k}
&=\frac{(i+j+k)!}{(i-1)!(j+1)!(k+1)!}i(j+k+1)\\
&=\lambda_{i,k}\lambda_{i+k,j}-\lambda_{k,j}\lambda_{i,j+k}.
\end{align*}
An easy induction proves that for any $i_1,\ldots,i_k \geq 1$:
\[\lambda(i_1,\ldots,i_k)=\frac{(i_1+\ldots+i_k)!}{i_1!\ldots i_k!}.\]
Hence following through the bijection, on the level of trees we get:
\begin{align*}
t_2&=\tdeux,\\
t_3&=\ttroisdeux+\ttroisun,\\
t_4&=\tquatrecinq+\tquatrequatre+3\tquatredeux+\tquatreun.
\end{align*}
Let us prove that for any tree $t$:
\[\mu(t)=\sum_{\mbox{\scriptsize $s$ leaf of $t$}} \mu(t\setminus s).\]
This is true if $t=\tun$, with the convention $\mu(1)=1$. Let us assume that the result is true
for any tree $t'$ such that $|t'|<[t|$. We put $t=B^+(t_1\ldots t_k)$, and $|t_i|=n_i$. Then:
\begin{align*}
\sum_{\mbox{\scriptsize $s$ leaf of $t$}} \mu(t\setminus s)
&=\sum_{i=1}^k \sum_{\mbox{\scriptsize $s$ leaf of $t_i$}} \mu(t\setminus s)\\
&=\sum_{i=1}^k \sum_{\mbox{\scriptsize $s$ leaf of $t_i$}} \frac{(n_1+\ldots+n_k-1)!}
{n_1!\ldots (n_i-1)!\ldots n_k!} \mu(t_1)\ldots \mu(t_i\setminus s)\ldots \mu(t_k)\\
&=\sum_{i=1}^k \frac{(n_1+\ldots+n_k-1)!}
{n_1!\ldots (n_i-1)!\ldots n_k!} \mu(t_1)\ldots\mu(t_{i-1})\mu(t_{i+1})\ldots \mu(t_k)
\sum_{\mbox{\scriptsize $s$ leaf of $t_i$}} \mu(t_i\setminus s)\\
&=\sum_{i=1}^k \frac{(n_1+\ldots+n_k-1)!}
{n_1!\ldots (n_i-1)!\ldots n_k!} \mu(t_1)\ldots\mu(t_{i-1})\mu(t_{i+1})\ldots \mu(t_k)\mu(t_i)\\
&=\mu(t_1)\ldots \mu(t_k)\frac{(n_1+\ldots+n_k-1)!}
{n_1!\ldots n_k!} \sum_{i=1}^k n_i\\
&=\mu(t_1)\ldots \mu(t_k)\frac{(n_1+\ldots+n_k)!}{n_1!\ldots n_k!}\\
&=\mu(t).
\end{align*}
\end{example}

\subsection{Further examples from a pre-Lie structure}

In \cite{Foissycontrole}, a pre-Lie product $\bullet$ is defined on $H_{CK}$ by the following:
if $f$ and $f'$ are forests,
\[f\bullet f'=\sum_{s\in Vert(f)}\mbox{grafting of $f'$ on $s$}.\]
This is a pre-Lie product which the reader will note is similar in nature to the product of $H_{GL}$. Moreover, for any $x,y,z\in H_{CK}$:
\begin{align*}
(xy)\bullet z&=(x\bullet z)y+x(y\bullet z),\\
\Delta(x\bullet y)&=x^{(1)}\otimes x^{(2)}\bullet y+x^{(1)}\bullet y^{(1)}\otimes x^{(2)}y^{(2)}.
\end{align*}

\begin{prop}
\label{prop::examples_from_prelie_structure}
Let $t'_1,\ldots,t'_N$ be a finite sequence of elements of $H_{CK}$ such that:
\begin{enumerate}
\item $t'_1=\tun$.
\item For all $i$, $t'_i$ is a nonzero linear combination of trees of degree $i$.
\item $A=\K[t'_1,\ldots,t'_N]$ is a Hopf subalgebra of $H_{CK}$.
\end{enumerate}
Let $X$ be a nonzero primitive element of $A$, homogeneous of degree $N$. We define a sequence $(t_n)_{n\geq 1}$ by:
\[t_n=\begin{cases}
t'_n\mbox{ if }n\leq N,\\
t_{n-N}\bullet X\mbox{ if }n>N.
\end{cases}\]
Then $(t_n)_{n\geq 1}\in \seq$. The associated double sequence is denoted by $(\lambda_{i,j})_{i,j\geq 1}$.
If $k+l>N$:
\[\lambda_{k,l}=\lambda_{k-N,l}+\lambda_{k,l-N}+\alpha k \delta_{l,N},\]
with the convention $\lambda_{i,j}=0$ if $i\leq 0$ or $j\leq 0$ and where $\alpha$ is the coefficient of $t_N$
in $X$. 
\end{prop}

\begin{proof}
As $X$ is nonzero and homogeneous of degree $N$, an easy induction proves that $t_n$ 
is a nonzero linear combination of trees of degree $n$ for any $n$.
We shall use the two following maps:
\begin{align*}
F&:\left\{\begin{array}{rcl}
H_{CK}&\longrightarrow&H_{CK}\\
x&\longrightarrow&x\bullet X,
\end{array}\right.&
D&:\left\{\begin{array}{rcl}
H_{CK}&\longrightarrow&H_{CK}\\
x&\longrightarrow&x\bullet 1.
\end{array}\right.
\end{align*}
Both are derivations of $H_{CK}$. Moreover, for any $n \geq 1$:
\begin{align*}
F(t_n)&=t_{n+N},&D(t_n)&=nt_n,
\end{align*}
where the latter identity holds because there are $n$ places in each tree of $t_n$ into which to graft, but grafting on $1$ leaves each tree unchanged, so we are left just counting vertices.
Consequently, $A$ is stable under $F$ and $D$. Moreover, for any $x\in H_{CK}$:
\begin{align*}
\Delta\circ N(x)&=\Delta(x\bullet X)\\
&=x^{(1)}\otimes x^{(2)}\bullet X+x^{(1)}\bullet X^{(1)}\otimes x^{(2)}X^{(2)}\\
&=x^{(1)}\otimes x^{(2)}\bullet X+x^{(1)}\bullet X\otimes x^{(2)}+x^{(1)}\bullet 1\otimes x^{(2)}X\\
&=(F\otimes Id +Id \otimes F)\circ \Delta(x)+(D\otimes Id)\circ \Delta(x)(1\otimes X). 
\end{align*}
Let us prove that $\Delta(t_n) \in A\otimes A$ by induction on $n$. If $n<N$, $t_n=t'_n$.
As $\K[t'_1,\ldots,t'_N]$ is a Hopf subalgebra of $H_{CK}$, $\Delta(t_n)\in \K[t'_1,\ldots,t'_N]^{\otimes 2}
\subseteq A^{\otimes 2}$. Otherwise, $t_n=F(t_{n-N})$. By the induction hypothesis, $\Delta(t_{n-N})\in A\otimes A$,
so:
\begin{align*}
\Delta(t_n)&\in (F\otimes Id +Id \otimes F)(A\otimes A)+(D\otimes Id)(A\otimes A)(1\otimes X).
\end{align*}
As $X\in A$ and $A$ is stable under $F$ and $D$, $\Delta(t_n) \in A\otimes A$. Finally, $(t_n)_{n\geq 1}\in \seq$.\\

By definition of $(\lambda_{i,j})_{i,j\geq 1}$, for any $n$:
\[\Delta(t_n)=t_n\otimes 1+1\otimes t_n+\sum_{i+j=n} \lambda_{i,j}t_i\otimes t_j+\mbox{nonlinear terms}.\]
Consequently:
\begin{align*}
&\Delta(t_{n+N})\\
&=\Delta\circ F(t_n)\\
&=t_{n+N}\otimes 1+1\otimes t_{n+N}+\sum_{i+j=n}\lambda_{i,j}(t_{N+i}\otimes t_j+t_i\otimes t_{N+j})
+nt_n\otimes \alpha t_N+\mbox{nonlinear terms}.
\end{align*}
Hence, if $k+l=n+N$, $\lambda_{k,l}=\lambda_{k-N,l}+\lambda_{k,l-N}+k\alpha \delta_{l,N}$. \end{proof}

\begin{example}\label{subexample}
As a subexample of the above,  take $t'_1=\tun$, $t'_2=\tdeux$ and $X=2\tdeux-\tun\tun$. Then:
\begin{align*}
t_1&=\tun,\\
t_2&=\tdeux,\\
t_3&=2\ttroisdeux-\ttroisun,\\
t_4&=2\tquatrecinq+2\tquatredeux-\tquatrequatre-\tquatreun,\\
t_5&=4\tcinqquatorze+4\tcinqdouze-2\tcinqtreize-2\tcinqdix-4\tcinqhuit+2\tcinqsix+4\tcinqcinq-4\tcinqdeux+\tcinqun.
\end{align*}
The associated double sequence $(\lambda_{i,j})_{i,j\geq 1}$ can be inductively computed by:
\begin{align*}
&&\lambda_{1,1}&=1,\\
&\mbox{If }k+l\geq 2,&\lambda_{k,l}&=\lambda_{k-2,l}+\lambda_{k,l-2}+2k\delta_{l,2}.
\end{align*}
This gives: 
\begin{enumerate}
\item If $k,l\geq 1$, $\displaystyle \lambda_{2k,2l}=4\binom{k+l}{l+1}$.
\item If $k\geq 0$ and $l\geq 1$, $\displaystyle\lambda_{2k+1,2l}=2\binom{k+l+1}{l+1}+2\binom{k+l}{l+1}$.
\item If $k,l\geq 0$, $\displaystyle\lambda_{2k+1,2l+1}=\binom{k+l}{l}$.
\item If $k\geq 1$ and $l\geq 0$, $\lambda_{2k,2l+1}=0$.
\end{enumerate}
\end{example}

Another example of a similar nature, but not included in the family of Proposition \ref{prop::examples_from_prelie_structure}, is the following.

\begin{example}\label{another example}
Let $a,b,c \in \K$. We consider the double sequence $(\lambda_{i,j})$ defined by:
\begin{itemize}
\item If $k,l \geq 1$, $\lambda_{2k,2l}=ak+b$.
\item If $k\geq 0$ and $l\geq 1$, $\displaystyle \lambda_{2k+1,2l}=ak+\frac{a+b}{2}$.
\item If $k\geq 1$ and $l\geq 0$, $\lambda_{2k,2l+1}=0$.
\item If $k,l\geq 0$, $\lambda_{2k+1,2l+1}=c$.
\end{itemize}
This family satisfies the pre-Lie relation, and it is an element of $\Lambda$ if, and only if $a+b\neq 0$ and $c\neq 0$.
The coefficient $\lambda(i_1,\ldots,i_k)$ only depends on $(p,q)$, where $p$ is the number of even $i_j$
and $q$ is the number of odd $i_k$: we shall denote by $\lambda'(p,q)$ their common value.
These coefficients can be inductively computed: if $p,q \geq 0$,
\begin{align*}
\lambda'(0,0)&=1,\\
\lambda'(p+1,0)&=\left(\frac{a+(3-2p)b}{2}\right)\lambda'(p,0),\\
\lambda'(p,q+1)&=\begin{cases}
c(\lambda'(p,q)-q\lambda'(p+1,q-1))\mbox{ if $q$ is odd},\\
-cq\lambda'(p+1,q-1)\mbox{ if $q$ is even}.
\end{cases}
\end{align*}
This gives:
\begin{align*}
t_1(a,b,c)&=\tun,\\
t_2(a,b,c)&=c\tdeux,\\
t_3(a,b,c)&=\frac{c(a+b)}{2}\ttroisdeux-\frac{c(a+b)}{4}\ttroisun,\\
t_4(a,b,c)&=\frac{c^2(a+b)}{2}\tquatrecinq+\frac{c^2(a+b)}{2}\tquatredeux-
\frac{c^2(a+b)}{4}\tquatrequatre-\frac{c^2(a+b)}{4}\tquatreun,\\
t_5(a,b,c)&=\frac{c^2(a+b)^2}{4}\tcinqquatorze+\frac{c^2(a+b)^2}{4}\tcinqdouze-\frac{c^2(a+b)^2}{8}
\tcinqtreize-\frac{c^2(a+b)^2}{8}\tcinqdix-\frac{c^2(a+b)^2}{4}\tcinqhuit\\
&+\frac{c^2(a+b)^2}{8}\tcinqsix
+\frac{c^2(a+b)(a-b)}{8}\tcinqcinq-\frac{c^2(a+b)(a-b)}{8}\tcinqdeux+\frac{c^2(a+b)(a-b)}{32}\tcinqun.
\end{align*}

An especially interesting case is $a=b$. If this holds, only a finite number of $\lambda'(p,q)$ are nonzero;
they are listed below.
\begin{align*}
\lambda'(1,0)&=a,&\lambda'(0,1)&=c,&\lambda'(1,1)&=ac,&\lambda'(0,2)&=-ac,&\lambda'(0,3)&=-3ac^2.
\end{align*}
\end{example}

\begin{remark}
Let us compare Example~\ref{another example} with Example~\ref{subexample}. We find:
\begin{align*}
t_1(a,b,c)&=t_1,\\
t_2(a,b,c)&=c t_2,\\
t_3(a,b,c)&=\frac{c(a+b)}{4} t_3,\\
t_4(a,b,c)&=\frac{c^2(a+b)}{4} t_4,\\
t_5(a,b,c)&=\frac{c^2(a+b)^2}{16}t_5-\frac{c^2(a+b)(a+3b)}{8}(\tcinqcinq+\tcinqdeux-4\tcinqun).
\end{align*}
Hence, $t_5(a,b,c)$ and $t_5$ are colinear if, and only if, $a+3b=0$.
A similar computation proves that $t_6(a,b,c)$ and $t_6$ are colinear if, and only if, $b=0$. 
Hence, as $a+b\neq 0$, $\K[t_n(a,b,c),n\geq 1]$ and $\K[t_n,n\geq 1]$ are different.
\end{remark}

\subsection{Definition of strong and weak sequences}

Suppose that $s = (t_n)_{n \geq 1}$ is in $\seq$. Then we can define nonzero Feynman rules $\phi$ as discussed in Section \ref{sect::higher_order_RGEs} to turn $X_s$ into a Green function to see if it is in fact the solution of a generalized renormalization group equation. 

\begin{defi}
\label{def::strong_sequence}
If a sequence $s \in \seq$ satisfies a $k$th order renormalization group equation for any choice of Feynman rules and for $\beta_1^{(k)} \neq 0$, then we say $s$ is a {\bf strong $k$th order sequence}.
\end{defi}

\begin{defi}
\label{def::weak_sequence}
If a sequence $s \in \seq$ satisfies a $k$th order renormalization group equation but is not strong, we say that it is a {\bf weak $k$th order sequence}.
\end{defi}

Sequences $s \in \seq$ that lead to strong $k$th order sequences have a unique order by which to be classified, as the conditions imposed on them mean that their order can be read off of the left-diagonal in the array depicted in Figure \ref{fig::array_of_coefficients}; this order is given by the order of the sequence $(\lam_{i,1})_{i \geq 1}$. Furthermore, since the leftward diagonals of Figure~\ref{fig::array_of_coefficients} must be sequences of order $k$ or lower for every nonzero choice of $\phi$, we see that the associated $\Lambda$-array must have all of its leftward diagonals as sequences of order $k$ or lower and the leftmost diagonal must be order exactly $k$ on account of the condition $\beta_1^{(k)} \neq 0$. Moreover any such $\Lambda$-array will give a strong $k$th order sequence.  Summarizing:

\begin{prop}
A sequence $s\in \seq$ is strong $k$th order if and only if its $\Lambda$-array $\theta(s)$ has all leftward diagonals of order at most $k$ and the leftmost diagonal of order exactly $k$.
\end{prop}

\begin{remark}\label{rem homog}
We can determine what condition on the $\Lambda$-array corresponds to the property that a strong $k$th order sequence's generalized renormalization group equations are homogeneous.  To this end fix an arbitrary choice of Feynman rules and consider the derivation of the fact that the diagonal sequences are order $k$ beginning on p\pageref{gRGE derivation}.  In this derivation, we are comparing coefficients in $\frac{\partial}{\partial L}G$ and $\overline{\beta}G$ for positive powers of $L$ and $x$.  If the generalized renormalization group equation is homogeneous then the coefficients of $\frac{\partial}{\partial L}G$ and $\overline{\beta}G$ must also match for $L^0$ and all powers of $x$.  Consequently for a homogeneous generalized group equation the triangular array, extended to include $c_{0,j}$ and $c_{i,0}$ entries, will still have all leftward diagonals of order $k$ or less.  

Translating over to $\seq$ and the $\Lambda$-arrays, $\lambda_{i,j}$ is the coefficient of $t_j\otimes t_i$ in $\Delta(t_{i+j})$ so extending the $\Lambda$-array to $0$th indexed entries corresponds to adding a $t_0$ term to the element of $\seq$ and the coefficients of $t_n\otimes t_0$ and $t_0\otimes t_n$ to the $\Lambda$-array. In the transition from the $c_{i,j}$ triangular arrays to the $\Lambda$-arrays we took the convention that elements of $\seq$ were normalized so that $t_1=1$.  Because of this we cannot assume that $t_0=1$, but only that $t_0$ is some nonzero constant, write it as $1/c$.  Since the coproducts of trees all have a primitive part the coefficients of  $t_n\otimes t_0$ and $t_0\otimes t_n$ in $\Delta(t_{n+1})$ are all $c$, as is the coefficient of $t_0\otimes t_0$ in $\Delta(t_0)$, so the entries of the new outer diagonals on the $\Lambda$-array are all $c$.

This tells us that given a $\Lambda$-array for a $k$th order strong sequence, if the generalized renormalization group equation for the sequence is homogeneous then the $\Lambda$-array has the property that if we add new outer diagonals of a constant value $c$ then the leftward diagonals of this enlarged array remain $k$th order or less.

In the reverse direction, if we can add such an outer layer, then returning one last time to the derivation of p\pageref{gRGE derivation}, rescaling $t_0$ to $1$ and $t_1$ to $c\bullet$ to match the conventions in place for that derivation, we see that the identities hold not only for positive coefficients of $L$ and $x$, but also for the coefficients involving $L^0$ and $x^0$, and so no $\gamma_0$ term is required in the generalized renormalization group equation.  That is, the generalized renormalization group equation is homogeneous.

In summary, let $s$ be a strong $k$th order sequence.  
The $\Lambda$-array of $s$ can be extended to have $0$ indexed entries all of value $c$ while maintaining the property that all leftward sequences are of order $k$ or less if and only if the generalized renormalization group equations satisfied by $s$ for each choice of Feynman rules are homogeneous.
\end{remark}

In Section \ref{sect::characterization_of_strong_sequences} we characterize strong $k$th order sequences.

Weak sequences are much more difficult to work with.  Note that there are two ways to be a weak $k$th order sequences, either $G(x,L) = \phi(X_s)$ satisfies a $k$th-order renormalization group equation for some (but not all) choices of Feynman rules, and/or $G(x,L) = \phi(X_s)$ satisfies a $k$th-order renormalization group equation such that $\beta_1^{(k)} = 0$. In the first case we have freedom to choose Feynman rules which get rid of problem parts of the $\Lambda$-array, and so weak $k$th order sequences can be much wilder.  We will comment on some nice cases and examples in Section~\ref{sect::comments_on_weak_sequences}.

\begin{example}\label{eg corolla 2}
The corollas (see Example~\ref{eg corolla 1}), because of all the $0$s in their $\Lambda$ array, can easily be scaled to generate strong sequences of any order.  Consider the sequence given by $t_n = (n!)^kc_n$ where the $c_n$ is the corolla on $n$ vertices.  Then the coefficient of $c_1\otimes c_{n-1}$ in $\Delta(t_n)$ is $n(n!)^k$, so the coefficient of $t_1\otimes t_{n-1}$ in $\Delta(t_n)$ is $n^{k+1}$.  This gives a leftmost diagonal in the $\Lambda$-array of order $k+1$ and all other diagonals remain $0$.  Therefore the sequence of these $t_n$ is a strong $(k+1)$st order sequence.
\end{example}

\section{Characterization of strong sequences}
\label{sect::characterization_of_strong_sequences}

\subsection{Strong $0$th order sequences}
\label{subsect::strong_0th_order_seq}

For each $n\geq 0$, let $\ell_n$ be the ladder with $n$ vertices as in Example~\ref{eg ladder 1}.
For any $n\geq 1$, we put
\begin{align*}
p_n&=\ln^*(\ell_n)\\
&=\sum_{k=1}^n \sum_{\substack{j_1+\ldots+j_k=n,\\ j_1,\ldots,j_k\geq 1}} 
\frac{(-1)^{k+1}}{k}\ell_{j_1}\ldots \ell_{j_k}\\
&=\sum_{1i_1+\ldots+ni_n=n}\frac{(-1)^{i_1+\ldots+i_n+1}}{i_1+\ldots+i_n} 
\frac{(i_1+\ldots+i_n)!}{i_1!\ldots i_n!}\ell_1^{i_1}\ldots \ell_n^{i_n},
\end{align*}
where $\ln^*$ is the convolution log. In a completion of $H_{CK}$, we put
\[X=\sum_{k\geq 1}\ell_k.\]
Then:
\[\sum_{k\geq 1}p_k=\ln^*(1+X)=\ln(1+X).\]
As $1+X$ is a group-like element, this is a primitive element, so for any $n\geq 1$,
its $n$-th homogeneous component is primitive; that is, the $p_n$ are primitive. 

\begin{prop}\label{prop z in seq}
Let $n\geq 1$ and let $b$ be a nonzero element of $\K$. We consider 
\[Z_n(b) :=\sum_{k\geq 1}z_k=B_+\left(\exp\left(\sum_{k=1}^{n-1}p_k+bp_n\right)\right).\]
Then $(z_k)_{k\geq 1}$ belongs to $\seq$. The corresponding family $(\lambda_{i,j})_{i,j\geq 1}$ is given by:
\[\lambda_{i,j}=\begin{cases}
1\mbox{ if }j<n,\\
b\mbox{ if }j=n,\\
0\mbox{ otherwise}.
\end{cases}\]
\end{prop}

\begin{proof}
We put $Y=p_1+\ldots+p_{n-1}+bp_n$ and let $Z$ denote the corresponding $Z_n(b)$. Now $Y$ is a primitive element, so:
\begin{align*}
\Delta(Y)&=Y\otimes 1+1\otimes Y,\\
\Delta(\exp(Y))&=\exp(Y\otimes 1+1\otimes Y)\\
&=\exp(Y\otimes 1)\exp(1\otimes Y)\\
&=\exp(Y)\otimes \exp(Y),\\
\Delta(Z)&=1\otimes Z+B^+(\exp(Y))\otimes \exp(Y)\\
&=1\otimes Z+Z\otimes \exp(Y).
\end{align*}
For any $k<n$, the $k$-th homogeneous component of $Y$ is
\begin{align*}
y_k&=\pi_k\left(\exp\left(p_1+\ldots+p_{n-1}+bp_n\right)\right)\\
&=\pi_k\left(\exp\left(\sum_{k\geq 1}p_k\right)\right)\\
&=\pi_k\left(\exp\left(\ln(1+X)\right)\right)\\
&=\pi_k(1 + X)\\
&=\ell_k.
\end{align*}
As $z_{k+1}=B^+(y_k)$ for any $k\geq 1$, for any $k\leq n$, $z_k=\ell_k$.
By construction, $Y\in \K[p_1,\ldots,p_n]=\K[\ell_1,\ldots,\ell_n]=\K[z_1,\ldots,z_n]$,
which implies that $Z\otimes \exp(Y)$ is an element of the completion of $\K[z_1,z_2,\ldots]\otimes \K[z_1,\ldots,z_n]$.
Hence, for any $n\geq 1$, $\Delta(z_n)\in \K[z_1,\ldots,z_n]^{\otimes 2}$,
and $(z_n)_{n\geq 1}\in \seq$.

Let us denote by $\varpi$ the canonical projection from $H_{CK}$ to the space of trees.
We obtain:
\begin{align*}
(\varpi\otimes \varpi)\circ \Delta(Z)&=\sum_{i,j\geq 1}\lambda_{i,j} z_i\otimes z_j\\
&=(\varpi\otimes \varpi)\left(1\otimes Z+Z\otimes \exp(Y)\right)\\
&=Z\otimes (z_1+\ldots+z_{n-1}+bz_n).
\end{align*}
Identifying, we obtain the formula for $\lambda_{i,j}$. \end{proof}

\begin{prop}\label{prop 0th order lambdas}
Let $(\lambda_{i,j})_{i,j\geq 1}\in \Lambda$ such that for any $j\geq 1$, the sequence $(\lambda_{i,j})_{i\geq 1}$ is constant. 
Up to a rescaling, we assume that $\lambda_{1,1}=1$. Then:
\begin{itemize}
\item either $\lambda_{i,j}=1$ for any $i,j\geq 1$,
\item or there exists $n\geq 1$ and $b\in \K\setminus\{1\}$ such that for any $i,j\geq 1$:
\[\lambda_{i,j}=\begin{cases}
1\mbox{ if }i<n,\\
b\mbox{ if }i=n,\\
0\mbox{ otherwise}.
\end{cases}\]
\end{itemize}\end{prop}

\begin{proof}
Let us consider the sequence $(a_j)_{j\geq 1}$ such that for any $i,j\geq 1$, $\lambda_{i,j}=a_j$.
The pre-Lie relation then gives, for any $i,j,k\geq 1$,
\[a_ja_k-a_ka_{j+k}=a_ka_j-a_ja_{j+k},\]
so $a_{j+k}(a_j-a_k)=0$. For $k=1$, as $a_1=1$, we obtain that for any $j\geq 1$,
either $a_j=1$ or $a_{j+1}=0$.

Let us assume that there exists $i\geq 1$, such that $a_i\neq 1$.  
Let us consider the minimal $n$ such that $a_n \geq 1$, and let us put $a_n=b$.
An easy induction proves that $a_m=0$ for any $m>n$. \end{proof}

\begin{theo}\label{thm 0th order}
The strong $0$th order sequences are the ladders and the $Z_n(b)$ of Proposition~\ref{prop z in seq}
\end{theo}

\begin{proof}
The ladders have all $\lambda_{i,j}=1$, so Propositions~\ref{prop z in seq} and \ref{prop 0th order lambdas} imply the theorem.
\end{proof}

Note that the set of $0$th order sequences consists of the ladders as noted explicitly in the proposition and elements of $\seq$ corresponding to the $Z_n(b)$ which interpolate between ladders and corollas. For example, the element of $\seq$ corresponding to $Z_0$ is exactly the sequence of $0$th-order corollas and the the sequence corresponding to $Z_n(b)$ consists of sums of trees which are corollas whose leaves have been replaced by ladders having at most $n + 1$ nodes. See Table~\ref{tab::0th_order_ex}.

Note also that by Remark~\ref{rem homog} the ladders have a homogeneous renormalization group equation for each choice of Feynman rules, but the $Z_n(b)$ can never have this homogeneity.  To see this, we take $t_0=1$ for the ladders, giving a larger array of all $1$s, while for the $Z_n(b)$ from the $j=1$ diagonal we must have $t_0=1$ but then for $j>n$ the diagonal sequence is $1,0,0,\ldots$ which is not $0$th order.

\begin{table}
\setlength{\tabcolsep}{4pt}

\begin{tabular}{|c|c|c|}
\hline
     \setlength{\tabcolsep}{2pt}
\begin{tabular}{ccccccccccccccc}
 & & & & & & & 1 & & & & & & &  \\ 
 & & & & & & 1 & & 1 & & & & & & \\  
 & & & & & 1 & & 1 & & $b$ & & & & & \\
 & & & & 1 & & 1 & & $b$ & & 0 & & & & \\
 & & & 1 & & 1 & & $b$ & & 0 & & 0 & & & \\
 & & 1 & & 1 & & $b$& & 0 & & 0 & & 0 & & \\
 & 1 & & 1 & &$b$ & & 0 & & 0 & & 0 & & 0 &\\
 1 & & 1 & & $b$ & & 0 & & 0 & & 0 & & 0 & & 0
\end{tabular}&   

$\begin{gathered}t_1 = \tun \\
  t_2 = \tdeux \\
  t_3 = \ttroisdeux \\
  t_4 = b\tquatrecinq - (b - 1)\tquatredeux + \frac{1}{3}(b - 1)\tquatreun\\
  t_5 = b\tcinqhuit + \frac{1}{2}\tcinqcinq - b\tcinqdeux + \frac{1}{3}(b - \frac{1}{4})\tcinqun
  \end{gathered}$ & $Z_2(b) = B^+(\exp(P_1 + P_2 + bP_3))$\\
 \hline 
\end{tabular}
      
    \caption{An example for $n=2$ and arbitrary $b$ of a $\Lambda$-array (left), its  corresponding sequence (middle), and its corresponding equation (right) from Proposition \ref{prop 0th order lambdas} and Theorem \ref{thm 0th order}.}
    \label{tab::0th_order_ex}
\end{table}

\subsection{Strong $1$st order  sequences}
\label{subsect::strong_1st_order_seq}

Next, we classify all strong  first-order sequences with the following theorem.

\begin{theo}
\label{thm::classification_1st_order_strong}
Let $(\lambda_{i,j} )_{i,j \geq 1} \in \Lambda$ be the pre-Lie sequence corresponding to a strong first-order element of $\seq$, and let $a_1 \in \K \setminus  \{0\}, a_2, b \in \K$. Then $(\lambda_{i,j} )_{i,j \geq 1} $ is one of the following:
\begin{itemize}
\item \underline{Case A:}
$$\lambda_{i,j} = \begin{cases}
        a_1i + b & \text{if $j = 1$}\\
        a_2i + b & \text{if $j = 2$}\\
        \dfrac{a_1 a_2i}{(j-1)a_1 - (j-2)a_2 } + b & \text{if $j \geq 3$} 
\end{cases}$$
\item \underline{Case B:}
$$\lambda_{i,j} = \begin{cases}
        a_2i - 2a_2 & \text{if $j = 2$}\\
        a_1i - 2a_1 & \text{otherwise}
\end{cases}$$
\item \underline{Case C:}
$$\lambda_{i,j} = \begin{cases}
        a_1i + 2a_1 & \text{if $j = 1$}\\
        a_2i + 4a_2 & \text{if $j = 2$}\\
        \dfrac{a_1 a_2(2j + i)}{\frac{(j-1)(j)(j+1)}{6}a_1 - \frac{(j-2)(j)(j+2)}{3}a_2 }  & \text{if $j \geq 3$} 
\end{cases}$$
\item \underline{Case D:}
$$\lambda_{i,j} = \begin{cases}
        a_1i + a_1 & \text{if $j = 1$}\\
        a_2i + 2a_2 & \text{if $j = 2$}\\
        \dfrac{a_1 a_2(i+j)}{\frac{(j-1)(j)}{2}a_1 - (j-2)(j)a_2 } & \text{if $j \geq 3$} 
\end{cases}$$
\item \underline{Case E:} $$\lambda_{i,j} = 
	\begin{cases}  a_1i + b  & \text{if $j = 1$} \\
				 0 & \text{otherwise}
	\end{cases}$$
\end{itemize}
\end{theo}
Note that the cases are not disjoint; in particular edge cases of one case may also appear in another case.

We remark in particular that this result encompasses strong first-order sequences that were already known. For example, Case E is that of first-order corollas, while setting $a_2 = a_1$ in Case A gives the sequences coming from Dyson-Schwinger type equations as in Example~\ref{eg::dyson_schwinger_example}  \cite{FoissyDyson}. Another special family arising from Case A will be discussed in Example \ref{ex::ladders_with_added_leaves}. 

Note also that by Remark~\ref{rem homog}, the corresponding generalized renormalization group equations are homogeneous in case $A$ when $b\neq 0$ using $t_0=1/b$ and cannot be homogeneous in case $A$ when $b=0$.  For cases $B$, $C$, and $D$, to have homogeneous generalized renormalization group equations the constant terms in $i$ for each value of $j$ must be the same, but for all three cases, the equality of the constant terms brings them to a special case of $A$.  Specifically, $B$ is homogeneous if $a_1=a_2$ which is the same as $A$ with $a_1=a_2$ and $b=-2a_2$, $C$ is homogeneous if $a_1=2a_2$ which is the same as $A$ with $b=2a_1$ and $a_1=2a_2$, and $D$ is homogeneous if $a_1=2a_2$ which is the same as $A$ with $b=a_1$ and $a_1=2a_2$.  Case $E$ can never be homogeneous as the $j=1$ diagonal gives $t_0=1/b$ and so requires $b\neq 0$ but this gives the sequence $b,0,0,0, \ldots$ for $j>1$ which is not of any order.

The proof of Theorem \ref{thm::classification_1st_order_strong} is straightforward, but requires the following technical lemma:

\begin{lemma}
\label{lem::technical_lemma_1}
If $(\lam_{i,j})_{i,j \geq 1}$ is a strong first-order sequence, then $\lam_{m,1} \neq \lam_{1,m}$ for any $m \geq 2$.
\end{lemma}

\begin{proof}
Let $(\lam_{i,j})_{i,j \geq 1}$  be a strong first-order sequence and assume towards a contradiction that there is an integer $m \geq 2$ such that $\lam_{m,1} = \lam_{1,m}$.  If this is the case, then for any $i,j \geq 1$, $PL(i,j,m)$ gives us that:
$$\lam_{i,1}\lam_{i+1, m}  = \lam_{i,m}\lam_{i + m, 1}$$
But since $(\lambda_{i,j} )_{i,j \geq 1}$ is a strong first-order sequence, it follows that $\lam_{i,j} = a_ji + b_j$ for some $a_j, b_j \in \K$, and so making this substitution we find:
\begin{align*}
(a_1i + b_1)(a_m(i+1) + b_m) = (a_mi + b_m)(a_1(i+m) + b_1)
\end{align*}
Distributing terms and then grouping by powers of $i$ we get:
$$(2a_mb_1 + 2a_1b_m)i + (a_1b_mm + a_mb_1) = 0$$
which gives the simultaneous conditions:
\begin{equation}
\label{eq::techinical_lemma_eq1}
a_mb_1 + a_1b_m = 0 
\end{equation}
and
\begin{equation}
\label{eq::techinical_lemma_eq2}
a_1b_mm + a_mb_1 = 0
\end{equation}
\underline{Case 1:} $b_1 = 0$. Since $(\lam_{i,j})_{i,j \geq 1}$ is first-order strong by hypothesis, $a_1 \neq 0$ and so equation (\ref{eq::techinical_lemma_eq1}) implies immediately that $b_m = 0$. Hence:
\begin{align*}
&\lam_{m,1} = \lam_{1,m} \\
\implies &ma_1 + b_1 = a_m + b+m \\
\implies &ma_1 = a_m
\end{align*}
However if we substitute $b_1 = b_m = 0$ and $ma_1 = a_m$ into $PL(1,1,m)$, we obtain: 
\begin{align*}
2ma_1^2 &= m(m+1)a_1^2\\
\implies a_1^2(1 - m) &= 0 
\end{align*}
which implies $a_1 = 0$, since $m \geq 2$. This is a contradiction.

\underline{Case 2:} If $b_1 \neq 0$, on the other hand, we can use equation (\ref{eq::techinical_lemma_eq2}) to solve for $a_m$:
$$a_m = \frac{-ma_1b_m}{b_1}$$
Substituting into equation \eqref{eq::techinical_lemma_eq1} gives:
\begin{align*}
\left(\frac{-ma_1b_m}{b_1} \right)b_1 + a_1b_m &= 0 \\
\implies a_1b_m(m - 1) &= 0
\end{align*}
And so $b_m = 0$. Further, substituting $b_m = 0$ into equation (\ref{eq::techinical_lemma_eq1}) gives that $a_m = 0$ as well. Taken together, this means that $\lam_{1,m} = 0$, and so by hypothesis $\lam_{m,1} = 0$.

Now consider the relation $PL(1,1,k-1)$:
$$\lam_{1,1}\lam_{2, m-1} - \lam_{1,m-1}\lam_{1,m} = \lam_{1,m-1}\lam_{m,1}  -  \lam_{m-1, 1}\lam_{1,m}$$
Since $\lam_{1,m} = \lam_{m,1}  = 0$, only the leftmost term does not vanish, leaving:
$$\lam_{1,1}\lam_{2, m-1}  =  0$$
Hence $\lam_{2, m-1}  =  0$. We then consider $PL(2,1, m-2)$, which after substituting in all calculated values gives:
$$\lam_{2,1}\lam_{3,m-2} = 0$$
In particular, since $\lam_{i,1}$ is linear in $i$ and $\lam_{m,1} = 0 $, then $\lam_{t,1} \neq 0$ for any $t \neq k$. Hence $\lam_{3, m-2} = 0$.

Proceeding inductively in this way, the relations $PL(4, 1, m-3), PL(5,1,m-4), \ldots , PL(m-2,1,2)$ imply that $\lam_{k_1, k_2} = 0$ for any $k_1, k_2$ such that $k_1 + k_2  = m+1$. Hence $(\lam_{i,j})_{i,j} \not \in \Lambda$ after all, since $(\lam_{i,j})_{i,j} $ does not satisfy the non-degeneracy condition of Theorem \ref{theo1}. This is a contradiction. 
\end{proof}

We can now proceed with the following proof of Theorem \ref{thm::classification_1st_order_strong}.

\begin{proof}[Proof of Theorem \ref{thm::classification_1st_order_strong}]
Let $(\lam_{i,j})_{i,j \geq 1} := (a_ji + b_j)_{i,j \geq 1}$ be a strong first-order sequence and consider the relations $PL(i,1,k)$ for arbitrary $i,k \geq 1$. These have the form:
$$\lam_{i,1}\lam_{i+1,k} - \lam_{1,k}\lam_{i,k+1} = \lam_{i,k}\lam_{i+k,1} - \lam_{k,1}\lam_{i, k+1}$$ 
Now since $\lam_{m,1} \neq \lam_{1,m}$ for any $m \geq 2$ by Lemma \ref{lem::technical_lemma_1} we can solve for $\lam_{i, k+1}$:
\begin{align*}
\lam_{i, k+1} &= \frac{\lam_{i,1}\lam_{i+1, k} - \lam_{i,k}\lam_{i+k, 1}}{\lam_{1,k} - \lam_{k,1}} \\
\implies a_{k+1}i + b_{k+1} &= \frac{a_1a_k(1 - k)i - a_1b_kk + a_kb_1}{(a_k + b_k) - (a_1k + b_1)} , \hspace{1cm} k \geq 2
\end{align*}
Hence for $k \geq 2$:
\begin{equation}
\label{eq::tech_lemma_ak}
a_{k+1} = \frac{a_1a_k(1 - k)}{(a_k + b_k) - (a_1k + b_1)}
\end{equation}
and 
\begin{equation}
\label{eq::tech_lemma_bk}
b_{k+1} = \frac{- a_1b_kk + a_kb_1}{(a_k + b_k) - (a_1k + b_1)}
\end{equation}
We will find a system of equations that constrain these variables even further. Using equations (\ref{eq::tech_lemma_ak}) and (\ref{eq::tech_lemma_bk}) we compute:
\begin{align*}
a_3 = \frac{-a_1a_2}{a_2+b_2-2a_1-b_1}\\
b_3 = \frac{b_1a_2-2b_2a_1}{a_2+b_2-2a_1-b_1}\\
\end{align*}
and corresponding explicit forms for $a_4, b_4, a_5, b_5, a_6, b_6, a_7$.
But these are not the only relations that exist among the variables. Using $PL(1,2,3), PL(1,2,4)$, and $PL(1,2,5)$ respectively, we calculate the following three values in a second way:
\begin{align*}
a_5' = \frac{-a_2a_3}{2a_3+b_3-3a_2-b_2}\\
a_6' = \frac{-2a_2a_4}{2a_4+b_4-4a_2-b_2}\\
a_7' = \frac{-3a_2a_5}{2a_5+b_5-5a_2-b_2} \\
\end{align*}
We may assume that the denominators in each of $a_5', a_6', a_7'$ cannot be $0$ for the following reason. Suppose that $2a_3+b_3-3a_2-b_2 = 0$. Then the corresponding relation $PL(1,2,3)$ would imply that $-a_2a_3 = 0$, so either $a_2 = 0$ or $a_3 = 0$. In either case, substituting the values into the equation for $a_3$ above gives $a_2 = a_3 = 0$. Applying equations (\ref{eq::tech_lemma_ak}) and (\ref{eq::tech_lemma_bk}) inductively then gives that $a_k = 0, b_k = 0$ for all $k \geq 2$, and hence the corresponding sequence is the sequence of first-order corollas (Case E). The same reasoning can also be used to show that $2a_4+b_4-4a_2-b_2 \neq 0$ and  $2a_5+b_5-5a_2-b_2 \neq 0$ unless the array belongs to the Case E family. 

Now setting $a_5 = a_5'$ , $a_7 = a_7'$, making appropriate substitutions, and rearranging gives:
\begin{equation}
-(2a_1^2a_2 + 9a_1a_2b_1 + a_2b_1^2 - 6a_1^2b_2)(b_1 - b_2) = 0
\end{equation}
and
\begin{equation}
-(172a_1^4a_2 + 840a_1^3a_2b_1 + 365a_1^2a_2b_1^2 + 60a_1a_2b_1^3 + 3a_2b_1^4 - 520a_1^4b_2 - 180a_1^3b_1b_2 - 20a_1^2b_1^2b_2)(b1 - b2) = 0
\end{equation} 
The first obvious solution to this system of equations is $b_1 = b_2$. Making this substitution and applying equation (\ref{eq::tech_lemma_bk}) inductively gives us that $b_1 = b_2  = b_k$ for all $k \geq 3$. Setting $b := b_1$ and applying equation (\ref{eq::tech_lemma_ak}) inductively then gives the Case A formula from the theorem statement. 

On the other hand, if $b_1 \neq b_2$, then we may divide both equations by $(b_1 - b_2)$. Moreover, our assumption that the sequence $(\lambda_{i,j})_{i,j \geq 0}$ is first-order means in particular that $a_1 \neq 0$. Hence there exists some field element $c \in \K$ such that $a_2 = ca_1$. Making this substitution into both equations gives us the following system:
\begin{equation}
\label{eq::theorem_proof1}
-(2a_1^2c + 9a_1b_1c + b_1^2c - 6a_1b_2)a_1 = 0
\end{equation}
and 
\begin{equation}
\label{eq::theorem_proof2}
-(172a_1^4c + 840a_1^3b_1c + 365a_1^2b_1^2c + 60a_1b_1^3c + 3b_1^4c - 520a_1^3b_2 - 180a_1^2b_1b_2 - 20a_1b_1^2b_2)a_1 = 0
\end{equation}
After dividing by $a_1$, it is a simple matter to use equation (\ref{eq::theorem_proof1}) to solve for $b_2$:
\begin{equation*}
b_2 = \frac{c}{6a_1}(2a_1^2 + 9a_1b_1 + b_1^2)
\end{equation*}
Substituting this value into equation (\ref{eq::theorem_proof2}) gives us:
\begin{equation}
\frac{1}{3}(4a_1^4c - 5a_1^2b_1^2c + b_1^4c) = 0
\end{equation}
and factoring:
\begin{equation}
\frac{1}{3}(2a_1 + b_1)(2a_1 - b_1)(a_1 + b_1)(a_1 - b_1)c = 0
\end{equation}
If $b_1 = -2a_1 $ then we inductively compute Case B. If $b_1 = 2a_1$ then we inductively compute Case C. If $b_1 = a_1$ then we inductively compute Case D. The solution $b_1 = -a_1$ is extraneous, as it contradicts the non-degeneracy condition of pre-lie arrays. Finally, the solution $c = 0$ inductively gives Case E once again.

All that remains is to verify that each family is indeed a pre-lie array for arbitrary choices of $a_1, a_2,$ and $b$, and that each family is well-defined (namely that the expressions appearing in the denominators of the Case A, C, and D arrays are guaranteed to be nonzero).

The computations to show that each family is pre-lie are straightforward, but quite long. To maintain readability we only include high-level details of the computation here. In particular, after accounting for symmetry of pre-lie relations in the indices $j$ and $k$ there are only four cases ($j = 2, k = 1; j = 1, k \geq 3; j = 2, k \geq 3; j \geq 3, k \geq 3$) for each of the families A, C, and D to prove that arrays in these families are indeed pre-lie, while there are only two cases ($j = 2, k \neq 2; j \neq 2, k \neq 2$) on the indices of arrays in family B to prove the pre-lie condition.  For example when $j = 2, k = 1$  in the Case A family we compute:

\begin{align*}
LHS &= \lambda_{i,j} \lambda_{i+j,k} - \lambda_{j,k}\lambda_{i, j+k} \\
    &= \lambda_{i,2} \lambda_{i+2,1} - \lambda_{2,1}\lambda_{i, 3} \\
    &= \squareparens{a_2i + b}\squareparens{a_1(i+2) + b} - \squareparens{2a_1 + b}\squareparens{\frac{a_1a_2i}{2a_1 - a_2} + b}\\
    &= a_1a_2i^2 + 2a_1a_2i + a_1bi + 2a_1b + a_2bi - \frac{2a_1^2a_2i}{2a_1 - a_2} - 2a_1b - \frac{a_1a_2bi}{2a_1 - a_2}\\
    &= a_1a_2i^2 + a_1bi + a_2bi - \frac{a_1a_2bi}{2a_1 - a_2} - \frac{2a_1^2a_2i - 2a_1a_2i(2a_1 - a_2)}{2a_1 - a_2}\\
    &= a_1a_2i^2 + a_1bi + a_2bi - \frac{a_1a_2bi}{2a_1 - a_2} - \frac{2a_1a_2^2i - 2a_1^2a_2i}{2a_1 - a_2}\\
    &= a_1a_2i^2 + a_1bi + a_2bi - \frac{a_1a_2bi}{2a_1 - a_2} - \frac{a_1a_2^2i - 2a_1^2a_2i + a_1a_2^2i}{2a_1 - a_2}\\
    &= a_1a_2i^2 + a_1bi + a_2bi - \frac{a_1a_2bi}{2a_1 - a_2} - \frac{a_1a_2^2i - a_1a_2i(2a_1 - a_2)}{2a_1 - a_2}\\
    &= a_1a_2i^2 + a_1bi + a_2bi - \frac{a_1a_2bi}{2a_1 - a_2} - \frac{a_1a_2^2i}{2a_1 - a_2} + a_1a_2i + a_2b - a_2b + b^2 - b^2\\
    &= a_1a_2i(i+1) + a_1bi + a_2b(i+1) + b^2 - \frac{a_1a_2^2i}{2a_1 - a_2} - a_2b - \frac{a_1a_2bi}{2a_1 - a_2}  -b^2\\
    &= \squareparens{a_1i + b}\squareparens{a_2(i+1) + b} - \squareparens{a_2 + b}\squareparens{\frac{a_1a_2i}{2a_1 - a_2} + b}\\
    &= \lambda_{i,1} \lambda_{i+1, 2} - \lambda_{1,2}\lambda_{i, 3} \\
    &= \lambda_{i,k} \lambda_{i+k, j} - \lambda_{k,j}\lambda_{i, k+j} \\
    &= RHS
\end{align*}

For the other cases, it is simply a matter of substituting in the formulas for the $\lambda_{i,j}$ into each term of equation \eqref{eq::prelie_relation}, moving all terms to one side of the equation, finding a common denominator, simplifying, and observing that the result is $0$. These steps can be easily performed in a computer algebra system such as Sage.

Finally, we only need to show that the arrays in each family are indeed well-defined. To maintain readability, we will only focus on the details for Case A, which is the content of the next lemma. 
\end{proof}

\begin{lemma}[Case $A$ array  is well-defined]
Let $\Lambda = (\lambda_{i,j})_{i,j \geq 1} := (a_ji + b_j)$ be a strong first order pre-lie array such that $b_1 = b_2$ (which implies $b_j = b_1 =: b$ for all $j$, and gives the Case $A$ family inductively). Then:
\begin{equation}
    (j-1)a_1 - (j-2)a_2 \neq 0 
\end{equation}
for any $j \geq 3$. 
\end{lemma}

\begin{proof}
\underline{Base Case:} ($j = 3$).

Assume towards a contradiction that $2a_1 - a_2 = 0$. Then in particular:
\begin{equation}
    a_2 = 2a_1
\end{equation}
Consider $PL(1,1,2)$:
\begin{align*}
    &\lambda_{1,1}\lambda_{2,2} - \lambda_{1,2}\lambda_{1,3} = \lambda_{1,2}\lambda_{3,1} - \lambda_{2,1}\lambda_{1,3}\\
    \implies &(a_1 + b)(2a_2 + b) - (a_2 + b)(a_3 + b) = (a_2 + b)(3a_1 + b) - (2a_1 + b)(a_3 + b)\\
    \implies &-a_1a_2 + 2a_1a_3 - a_2a_3 = 0 
\end{align*}
Substituting in $a_2 = 2a_1$ we obtain:
\begin{align*}
    & -2a_1^2 = 0\\
    \implies & a_1 = 0
\end{align*}
contradicting the fact that $a_1 \neq 0$ by virtue of the sequence being first order.

\underline{Inductive Step:} Let $m > 3$ be the least positive integer for which
\begin{align}
   & (m-1)a_1 - (m-2)a_2 = 0\label{eq::lemma_towards_contradiction1}\\
    \implies & \frac{(m-1)}{(m-2)}a_1 = a_2 \label{eq::lemma_towards_contradiction2}
\end{align}

and consider $PL(1,1,m-1)$:
\begin{align}
    &\lambda_{1,1}\lambda_{2,m-1} - \lambda_{1, m-1}\lambda_{1,m} = \lambda_{1,m-1}\lambda_{m,1} - \lambda_{m-1, 1}\lambda_{1,m} \nonumber\\
    \implies & (a_1 + b)(2a_{m-1} + b) - (a_{m-1} + b)(a_m + b) = (a_{m-1} + b)(ma_1 + b) - ((m-1)a_1 + b)(a_m + b) \nonumber\\
    \implies & (m-1)a_1a_m - (m - 2)a_1a_{m-1} - a_ma_{m-1} = 0 \label{eq::equation1}
\end{align}
Now since $m$ is the least positive integer for which equation (\ref{eq::lemma_towards_contradiction1}) holds, we can solve for $a_{m-1}$ inductively as in the formula for the Case $A$ array (this is our inductive hypothesis):
\begin{equation}
    a_{m-1} = \frac{a_1a_2}{(m-2)a_1 - (m-3)a_2}
\end{equation}
Substituting this into (\ref{eq::equation1}) we retrieve:
\begin{equation}
    (m-1)a_1a_m - (m-2)a_1\left[ \frac{a_1a_2}{(m-2)a_1 - (m-3)a_2}\right] - a_m\left[ \frac{a_1a_2}{(m-2)a_1 - (m-3)a_2}\right] = 0
\end{equation}
And finally, substituting the value of $a_2$ from (\ref{eq::lemma_towards_contradiction2}) we get:
\begin{align}
    &(m-1)a_1a_m - (m-2)a_1^2[m-1] - a_ma_1[m-1] = 0 \nonumber\\
    \implies & -a_1^2(m-2)(m-1) = 0 \label{eq::lemma_final_equation}\\
    \implies & a_1 = 0 \hspace{2cm} \text{(Since $m > 2$)} \nonumber
\end{align}
once again contradicting that the sequence is a first-order sequence.
\end{proof}

We remark that proving Cases C and D are well-defined goes through in exactly the same way as Case A, even down to using the same pre-lie relations. In place of equation (\ref{eq::lemma_final_equation}) one arrives at
\begin{equation*}
-a_1^2\frac{(m+1)(2m+1)(m-1)}{2m-1} = 0
\end{equation*}
in the proof of Case C and 
\begin{equation*}
-a_1^2(m+1)(m-1) = 0
\end{equation*}
in the proof of Case D. 

\begin{example}
\label{ex::ladders_with_added_leaves}

One special subcase of the Case A family has a nice combinatorial interpretation.

Consider rooted trees which are ladders with added leaves, that is if all leaves are removed the remaining tree is a ladder.  Weight such a tree $t$ by $\lambda_t$ defined by
\[\lambda_t=\begin{cases}
1\mbox{ if }t=\tun,\\
(a+b)a^{lea(t)-1}b^{depth(t)-1}\mbox{ if $t$ is a ladder with added leaves, different from $\tun$},\\
0\mbox{ otherwise},
\end{cases}\]
where $lea(t)$ is the number of leaves of $t$ and $depth(t)$ is the depth of $t$.
Collecting all such trees with these weights into the sum $X=\sum t_n$, then this sequence can be generated by 
\[X=B^+\left(\frac{1+bX}{1-a\tun}\right).\]
This can be seen by noting that the numerator inside $B_+$ generates the ladder backbone and the $1-a\tun$ denominator gives a geometric series in $a\tun$ generating the additional leaves at the root; the other additional leaves being generated recursively with the ladder backbone in the same way.

The associated $\Lambda$ array for this example is 
\[\lambda_{i,j}=\begin{cases}
ai+b\mbox{ if }j=1,\\
b\mbox{ if }j\geq 2.
\end{cases}\]
This can be seen by direct computation, see Lemma 5.9 of \cite{Dmmath} for details. 

\end{example}

\subsection{Strong $2$nd order and higher}
\label{subsect::2nd_order_seq}

Given the large number of examples of strong first-order sequences given in the last subsection, the results of this section may come as a surprise. The main result of this section is the following:

\begin{theo}[Classification of Strong $\ell$th Order Sequences]
\label{thm::classification_of_strong_kth_order_sequences}
For $\ell \geq 2$, the only family of strong $\ell$th order sequences is the family of scaled corollas. 
\end{theo}

Note that by Remark~\ref{rem homog} none of the strong $\ell$th order sequences for $\ell \geq 2$ can correspond to homogeneous generalized renormalization group equations since, by the same argument as for the corollas in the strong first order case, the $j=1$ diagonal gives $t_0=1/b$ and so requires $b\neq 0$ but this gives the sequence $b,0,0,0, \ldots$ for $j>1$ which is not of any order.

\begin{proof}

Let $\ell \geq 2$ be given, and consider the pre-Lie array $(\lambda_{i,j})_{i,j \geq 1}$ given by $\lam_{i,j} = f_j(i)$ where we define 
\begin{equation*}
    f_j(i) = a_{j,1}i^{\ell} + a_{j,2}i^{\ell - 1} + \ldots + a_{j,\ell}i + a_{j, \ell + 1}
\end{equation*}
Consider an arbitrary (and fixed) $k \geq 2$, and form the sequence of relations $\C = \set{PL(i,1,k) : i \in \N}$. Now we can ask ourselves what an arbitrary member of this set looks like? We simply  evaluate the pre-Lie relation of Theorem \ref{theo1} (where we have moved all terms to the left-hand side of the equation) at the appropriate indices and find that for arbitrary $i$:
\begin{equation}
\label{eq::general_prelie_relation_in_proof_of_strong_kth_order_theorem}
    PL(i,1,k): \hspace{1cm} f_j(i)f_{k}(i+j) - f_k(j)f_{j + k}(i) - f_{k}(i)f_j(i+k) + f_j(k)f_{k + j}(i) = 0
\end{equation}
We remark that the collection $\C$ has infinitely many relations, but only $3(k+1)$ variables since $k$ is fixed. Ultimately it is this over-saturation of equations that will enable us to prove the result.

Evaluating the $f_j$ in (\ref{eq::general_prelie_relation_in_proof_of_strong_kth_order_theorem}) above with the indicated arguments yields:
\begin{multline}
\label{eq::general_prelie_relation_in_the_proof_of_strong_kth_order_theorem_02}
    (a_{1,1}i^{\ell} + a_{1,2}i^{\ell - 1} + \ldots + a_{1,\ell}i + a_{1, \ell + 1})(a_{k,1}(i+1)^{\ell} + a_{k,2}(i+1)^{\ell - 1} + \ldots + a_{k,\ell}(i+1) + a_{k, \ell + 1}) \\
    - (a_{k,1} + a_{k,2} + \ldots + a_{k,\ell} + a_{k, \ell + 1} )(a_{k+1, 1}i^{\ell} + a_{k+1, 2}i^{\ell - 1} + \ldots + a_{k_o + 1, \ell}i + a_{k + 1, \ell + 1}) \\
    - (a_{k,1}i^{\ell} + a_{k,2}i^{\ell - 1} + \ldots + a_{k,\ell}i + a_{k, \ell + 1})(a_{1,1}(k + i)^{\ell} + a_{1,2}(k + i)^{\ell - 1} + \ldots + a_{1,\ell}(k + i) + a_{k, \ell + 1}) \\
    + (a_{1,1}k^{\ell} + a_{1,2}k^{\ell - 1} + \ldots + a_{1,\ell}k + a_{1, \ell + 1})(a_{k + 1,1}i^{\ell} + a_{k + 1,2}i^{\ell - 1} + \ldots + a_{k + 1,\ell}i + a_{k + 1, \ell + 1}) = 0
\end{multline}
Now the expression on the left hand side of (\ref{eq::general_prelie_relation_in_the_proof_of_strong_kth_order_theorem_02}) is a sum of products of polynomials in the variable $i$, and as a consequence is also a polynomial in $i$. Let us examine what the highest-degree term of (\ref{eq::general_prelie_relation_in_the_proof_of_strong_kth_order_theorem_02}) looks like. The highest power of $i$ appearing in (\ref{eq::general_prelie_relation_in_the_proof_of_strong_kth_order_theorem_02}) will be $2\ell$, with the first term contributing a term of $a_{1,1}a_{k,1}i^{2\ell}$, and the third term contributing a term of $-a_{1,1}a_{k,1}i^{2\ell}$, so taken together the coefficient of $i^{2\ell}$ is zero; the second and fourth terms will only contribute to the highest power of the polynomial when $2\ell = \ell$---that is, when $\ell = 0$ and $(\lambda_{i,j})_{i,j \geq 1}$ is $0$th-order. 

Hence we look instead at the next highest power of $i$, which is $2\ell - 1$. From the first term of (\ref{eq::general_prelie_relation_in_the_proof_of_strong_kth_order_theorem_02}), we obtain a factor of $(a_{1,1}a_{k,1}\ell + a_{1,1}a_{k,2} + a_{1,2}a_{k,1})i^{2\ell - 1}$. From the third term of (\ref{eq::general_prelie_relation_in_the_proof_of_strong_kth_order_theorem_02}) we obtain a factor $-(a_{1,1}a_{k,1}k\ell + a_{1,2}a_{k,1} + a_{1,1}a_{k,2})i^{2\ell - 1}$. Note that the second and fourth term of (\ref{eq::general_prelie_relation_in_the_proof_of_strong_kth_order_theorem_02}) will contribute to the highest power of $i$ only when $2\ell - 1 = \ell$---that is, when $\ell = 1$ and the sequence is consequently first order! Hence since $\ell  \geq 2$ we have that the coefficient of the highest power of $i$ is equal to:
\begin{align*}
    (a_{1,1}a_{k,1}\ell + a_{1,1}a_{k,2} + a_{1,2}a_{k,1})i^{2\ell - 1} - (a_{1,1}a_{k,1}k\ell + a_{1,2}a_{k,1} + a_{1,1}a_{k,2})i^{2\ell - 1} 
\end{align*}
which simplifies to:
\begin{equation}
    a_{1,1}a_{k,1}\ell(1 - k)
\end{equation}
Now the power of this setup comes down to the fact that  $i$ is an index ranging over the positive integers. This means that---by definition---the sequence of relations $\C$ is in fact a $(2\ell - 1)$th-order sequence in the variable $i$. But in turn, this means that taking the $(2\ell - 1)$st consecutive differences of the sequence will be constant, and hence that taking the $(2\ell)$th consecutive differences of $\C$ will be equal to $0$. Now taking the first difference will cause the $i^0$ (that is, constant) terms to cancel, the second difference will cause the $i$ terms to cancel, and so on. Hence by taking the $(2\ell)$th consecutive differences, every term up through the terms containing $i^{2\ell -1}$ will cancel. This means that we are left with the equation:
\begin{equation}
    Na_{1,1}a_{k,1}\ell(1 - k) = 0
\end{equation}
for $N$ a nonzero element of $\K$. (More precisely, $\displaystyle N = \sum_{j = 0}^{\ell}\binom{\ell}{j}(i - j)^{\ell}$. A standard combinatorial exercise can be used to show that this is equal to $\ell!$). If $a_{1,1} = 0$, it follows that the sequence $(\lambda_{i,j})_{i,j \geq 1}$ is not actually $\ell$th order. Moreover, we know that $\ell \geq 2$ by hypothesis, and $k \neq 1$ (as otherwise the pre-Lie relations are a tautology). Hence the only solution is that $a_{k,1} = 0$ for all $k \geq 2$.

We now perform induction on the second index of the $a_{k,m}$, taking as our base case the analysis above wherein $m = 1$. Suppose that $a_{k,t} = 0$ for all $t$ from $1$ up to $m-1$, and consider the case of $t = m$. (We are making the assumption that $m \leq \ell + 1$, as otherwise $a_{k,m} = 0$ already). Evaluating the $f_j$ in (\ref{eq::general_prelie_relation_in_proof_of_strong_kth_order_theorem}) with the correct values as we did before yields:
\begin{multline}
\label{eq::general_prelie_relation_in_the_proof_of_strong_kth_order_theorem_03}
    (a_{1,1}i^{\ell} + a_{1,2}i^{\ell - 1} + \ldots + a_{1,\ell}i + a_{1, \ell + 1})(a_{k,m}(i+1)^{\ell - m + 1} + a_{k,m+1}(i+1)^{\ell - m} + \ldots + a_{k,\ell}(i+1) + a_{k, \ell + 1}) \\
    - (a_{k,m} + a_{k,m+1} + \ldots + a_{k,\ell} + a_{k, \ell + 1} )(a_{k+1, m}i^{\ell - m + 1} + a_{k+1, m+1}i^{\ell - m} + \ldots + a_{k + 1, \ell}i + a_{k + 1, \ell + 1}) \\
    - (a_{k,m}i^{\ell - m +1} + a_{k,m + 1}i^{\ell - m} + \ldots + a_{k,\ell}i + a_{k, \ell + 1})(a_{1,1}(k + i)^{\ell} + a_{1,2}(k + i)^{\ell - 1} + \ldots + a_{1,\ell}(k + i) + a_{k, \ell + 1}) \\
    + (a_{1,1}k^{\ell} + a_{1,2}k^{\ell - 1} + \ldots + a_{1,\ell}k + a_{1, \ell + 1})(a_{k + 1,m}i^{\ell - m + 1} + a_{k + 1,m+1}i^{\ell - m} + \ldots + a_{k + 1,\ell}i + a_{k + 1, \ell + 1}) = 0
\end{multline}
Now this time, the highest power of $i$ appearing is $2\ell - m + 1$. However, we find that the first term contributes a term of $a_{1,1}a_{k,m}i^{2\ell - m + 1}$ and the third term contributes $-a_{1,1}a_{k,m}i^{2\ell - m + 1}$. Exactly as before, we get that the second and fourth terms contribute to the highest power of $i$ only when $2\ell - m + 1 = \ell - m + 1$; that is, exactly when $\ell = 0$ and $(\lambda_{i,j})_{i,j \geq 1}$ is a $0$th-order array. 

Hence we look at the second highest power of $i$, namely $i^{2\ell - m}$. From the first term of (\ref{eq::general_prelie_relation_in_the_proof_of_strong_kth_order_theorem_03}) we get a contribution of $(a_{1,1}a_{k,m}(\ell - m + 1) + a_{1,1}a_{k,m+1} + a_{1,2}a_{k,m})i^{2\ell - m}$ where we get $(\ell - m + 1)$ from the binomial theorem. From the third term of (\ref{eq::general_prelie_relation_in_the_proof_of_strong_kth_order_theorem_03}), we get a contribution of $-(a_{1,1}a_{k,m}k\ell + a_{1,2}a_{k,m} + a_{1,1}a_{k,m+1})i^{2\ell - m}$. We note once again that the second and fourth terms of (\ref{eq::general_prelie_relation_in_the_proof_of_strong_kth_order_theorem_03}) contribute only when $2\ell - m = \ell - m + 1$; that is, when $\ell = 1$. Hence the $i^{2\ell - m}$ term of (\ref{eq::general_prelie_relation_in_the_proof_of_strong_kth_order_theorem_03}) has a coefficient of
\begin{equation*}
    (a_{1,1}a_{k,m}(\ell - m + 1) + a_{1,1}a_{k,m+1} + a_{1,2}a_{k,m}) - (a_{1,1}a_{k,m}k\ell + a_{1,2}a_{k,m} + a_{1,1}a_{k,m+1})
\end{equation*}
which simplifies to:
\begin{equation*}
    a_{1,1}a_{k,m}(\ell - m + 1 - k\ell)
\end{equation*}
We take the $(2\ell - m + 1)$th consecutive differences of the sequence $\C$ to get an equation: 
\begin{equation}
    \label{eq::kth_order_strong_proof_inductive_step}
    Na_{1,1}a_{k,m}(\ell - m + 1 - k\ell) = 0
\end{equation}
for $N$ a nonzero constant; see the footnote included in the argument for the base case above. As before, $a_{1,1} \neq 0$, as otherwise $(\lambda_{i,j})_{i,j \geq 1}$ is not of order $\ell$. Now suppose that $\ell - m + 1 - k\ell = 0$. This would mean in particular that $\ell = m - 1 + k\ell$, but since $k \geq 2$, and $\ell \geq 2$, we must then have that $m < 0$. But this violates the inductive hypothesis. Hence we achieve that $\ell - m + 1 - k\ell \neq 0$, and so the only possibility in equation (\ref{eq::kth_order_strong_proof_inductive_step}) is that $a_{k, m} = 0$. Since $k \geq 2$ was chosen arbitrarily, it follows that $a_{k,m} = 0$ for all $k \geq 2$.

This completes the proof.
\end{proof}

\section{Comments on weak sequences}

\label{sect::comments_on_weak_sequences}

Perhaps the most interesting example of a weak sequence is the Connes-Moscovici Hopf algebra, see Examples~\ref{ex::Connes_Moscovici} and \ref{eg::connes_moscovici_example}.  As observed before the Connes-Moscovici generators are not a strong $k$-th order sequences for any $k$.  However, the $\lambda_{n, 1}$ sequence is a second order sequence and setting $\sigma(t)=0$ for $t\neq \tun$ we can set all the other sequences to $0$ and we see that the Connes-Moscovici Hopf algebra is a weak second order sequence.  Setting $\sigma(t)=0$ for $t \neq \tun$ is the same as taking the Feynman rules $\phi(t) = L^{|t|}/t!$, the tree factorial Feynman rules of Section \ref{sect::trees_and_tree_feynman_rules}.

This is nice because it is an example which is in some sense naturally second order, which we do not see among the strong sequences.  It emphasizes the difference between Connes-Moscovici and the Hopf subalgebras which come from Dyson-Schwinger equations, all of which are strong 1st order or strong 0th order.  The fact that Connes-Moscovici is second order, at least in this weak sense, was part of the motivation for defining higher order renormalization group equations.

\medskip

Unsurprisingly weak sequences are too wild for any useful characterization.  As a further example let us consider one particular, but still quite intricate case.  If we set all the $\lambda_{i,1}$ to $1$ then we have a constant sequence in the leftmost diagonal, so this will be weak $0$th order using the tree factorial Feynman rules $\phi(t) = L^{|t|}/t!$.  Even all such $0$th order sequences are hard to characterize.  However, if we restrict the rightmost diagonal to have entries only in $\{0,1\}$ then we are able to characterize the resulting sequences as detailed below.

We now look for all sequences $(\lambda_{i,j})_{i,j\geq 1}$ of $\seq$ such that :
\begin{enumerate}
\item For any $i\geq 1$, $\lambda_{i,1}=1$.
\item For any $j\geq 1$, $\lambda_{1,j}\in \{0,1\}$.
\end{enumerate}
The set of these sequences is denoted by $\seq_{0-1}$. 
If $(\lambda_{i,j})_{i,j\geq 1}$ is such a sequence, we put $\lambda_{1,j}=x_j$ for any $j\geq 1$. 
For any $j$, $x_j\in \{0,1\}$ and $x_1=1$. One of the reasons these sequences are of interest is that---by stipulating that the leftmost diagonal $(\lambda_{i,1})_{i \geq 1}$ is all $1$'s---the entire $\Lambda$-array is determined by a choice of the rightmost diagonal $(\lambda_{1,j})_{j \geq 1} =: (x_j)_{j \geq 1}$. By further assuming that $x_j \in \{ 0,1\}$ for each $j$, we can completely classify which choices of $(x_j)_{j \geq 1}$ give rise to a valid element of $\Lambda$. The following lemmas each prove a new structural characteristic of this sequence $(x_j)_{j \geq 1}$, culminating in the full classification, Proposition \ref{prop seq 01}.  \\

Let us consider a sequence $(\lambda_{i,j})_{i,j\geq 1}\in \seq_{0-1}$.

\begin{lemma}\label{lemma4}
Let $j,N\geq 1$ such that $x_j=\ldots=x_{j+N-1}=0$ and $x_{j+N}=1$. 
For any $p\in \{0,\ldots, N\}$,
\[\lambda_{i,j+p}=(-1)^{N-p}\binom{i-1}{N-p}.\]
\end{lemma}

\begin{proof}
By the pre-Lie relation $PL(i,j,1)$, for any $i\geq 1$, $\lambda_{i,j}-\lambda_{i,j+1}=\lambda_{i+1,j}-x_j \lambda_{i,j+1}$, so:
\begin{align}
\label{EQ2} \lambda_{i+1,j}&=\lambda_{i,j}+(x_j-1)\lambda_{i,j+1}.
\end{align}
We proceed by induction on $N-p$. If $N-p=0$, we proceed by induction on $i$. If $i=1$,
then $\lambda_{i,j+N}=x_{j+N}=1$. Otherwise, by the induction hypothesis, (\ref{EQ2}) implies that:
\[\lambda_{i,j+N}=\lambda_{i-1,j}+0=1,\]
so the results holds for any $i$ if $p=N$. If $N-p\geq 1$, we obtain that for any $i$:
\begin{align*}
\lambda_{i+1,j+p}&=\lambda_{i,j+p}-(-1)^{N-p-1}\binom{i-1}{N-p-1}\\
&=\lambda_{i,j+p}+(-1)^{N-p}\binom{i-1}{N-p-1}.
\end{align*}
In this case as well we proceed by induction on $i$. If $i=1$, then $\lambda_{1,j+N-p}=0$. Otherwise,
\begin{align*}
\lambda_{i,j+p}&=\lambda_{i-1,j+p}+(-1)^{N-p}\binom{i-2}{N-p-1}\\
&=(-1)^{N-p}\left(\binom{i-2}{N-p}+\binom{i-2}{N-p-1}\right)\\
&=(-1)^{N-p}\binom{i-1}{N-p}.
\end{align*}
So the result hold for any $i,j$. \end{proof}

\begin{lemma}
Let us assume that there exists $j,N\geq 2$  such that $x_j=\ldots=x_{j+N-1}=0$ and $x_{j+N}=1$.
Then $x_2=x_3=0$.
\end{lemma}

\begin{proof}
By the pre-Lie relation $PL(i,j,2)$, for any $i$:
\begin{align*}
(-1)^N\binom{i-1}{N}\lambda_{i+j,2}-(-1)^{N-2}\lambda_{j,2}\binom{i-1}{N-2}
&=(-1)^N\lambda_{i,2}\binom{1+i}{N}-(-1)^{N+N-2}\binom{1}{N}\binom{i-1}{N-2},\\
\binom{i-1}{N}\lambda_{i+j,2}-\lambda_{j,2}\binom{i-1}{N-2}&=\binom{1+i}{N}\lambda_{i,2}.
\end{align*}
For $i=N-1$, we obtain $\lambda_{j,2}=-\lambda_{N-1,2}$.  For $i=N$, we obtain $\lambda_{j,2}(N-1)=(N+1)\lambda_{N,2}$.
Finally:
\[\lambda_{j,2}=-\frac{N+1}{N-1}\lambda_{N,2}.\]
If $x_2=1$, by  Lemma \ref{lemma4}, for any $i$, $\lambda_{i,2}=1$: this is a contradiction. So $x_2=0$. 
If $x_3=1$, by Lemma \ref{lemma4}, for any $i$, $\lambda_{i,2}=i-1$ for any $i$. This gives
\[1-j=-\frac{N+1}{N-1}(1-N),\]
so $j=-N$, which is absurd. So $x_3=0$. \end{proof}

\begin{lemma}
Let us assume that $x_2=x_3=0$. For any $i\geq 1$, for any $j\geq 2$, $\lambda_{i,j}=0$. 
\end{lemma}

\begin{proof}
We first prove that $x_j=0$ for any $j$. Let us assume that there exists $N$ such that $x_N=1$.
We take the smallest possible $N$. By hypothesis, $N\geq 4$.

Let us assume that $N=4$. Then $PL(1,2,3)$ gives $2x_5=0$, so $x_5=0$.
Moreover, $PL(1,2,4)$ gives $2x_6-6=0$, so $x_6=3$: absurd, $x_3\in \{0,1\}$. So $x_4=0$.

Let us assume that $N=5$. Then $PL(1,3,4)$ gives $5x_7=0$, so $x_7=0$.
Moreover, $PL(2,2,4)$ gives $-10=0$, which is absurd. So $x_5=0$.

Let us assume that $N=6$. Then $PL(1,3,4)$ gives $-2x_7=0$, so $x_7=0$.
Moreover, $PL(2,3,4)$ gives $2x_8=0$, so $x_8=0$.
Finally, $PL(3,3,4)$ gives $-2x_9+20=0$, so $x_9=10$: absurd, $x_9\in \{0,1\}$. So $x_6=0$.\\

Let us assume that $N\geq 7$. 
If $N=2k$ is odd, we put $p=q=k+1$. Then $p,q\geq 4$ and:
\[p+q-3=2k-1<N<2k+1=p+q-1,\]
If $N=2k+1$ is odd, we put $p=q=k+1$. Then $p,q\geq 4$ and:
\[p+q-3=2k-1<N \leq 2k+1=p+q-1.\]
In both cases, $p,q\geq 4$ and $p+q-3<N\leq p+q-1$. The pre-Lie relation $PL(p,q,2)$ gives:
\begin{align*}
\lambda_{p,q}\lambda_{p+q,2}-\lambda_{q,2}\lambda_{p,q+2}&=\lambda_{p,2}\lambda_{p+2,q}-\lambda_{2,q}\lambda_{p,q+2}.
\end{align*}
By Lemma \ref{lemma4}, $\displaystyle \lambda_{p,2}=(-1)^{N-2}\binom{p-1}{N-2}$.
As $(N-2)-(p-1)>q-4\geq 0$, this is zero. Similarly, $\lambda_{q,2}=0$. Moreover,
$\displaystyle \lambda_{2,q}=(-1)^{N-q}\binom{1}{N-q}$. As $N-q>p-3\geq 1$, this is zero. We obtain:
\[\lambda_{p,q}\lambda_{p+q,2}=0.\]
By Lemma \ref{lemma4}, $\displaystyle \lambda_{p,q}=(-1)^{N-q}\binom{p}{N-q}$. As $p-N+q\geq 1$, this is nonzero.
Moreover, $\displaystyle \lambda_{p+q,2}=(-1)^{N-2}\binom{p+q-1}{N-2}$. As $(p+q-1)-(N-2)=p+q+1-N\geq 2$,
this is nonzero. This is a contradiction, so for any $j\geq 2$, $x_j=0$.\\

A direct induction on $i$ using (\ref{EQ2}) proves that $\lambda_{i,j}=0$ if $j\geq 2$. \end{proof}

Therefore, the sequence $(x_i)_{i\geq 1}$ has the following form:
\[(1,1,\ldots,1,0,1,\ldots,1,0,1,\ldots) \mbox{ or }
(1,1,\ldots,1,0,1,\ldots,1,0,1,\ldots, 1,0,0,\ldots).\]
We denote by $J$ the set of the indices of the  isolated $0$ in the sequence $(x_j)_{j\geq 1}$ and by $K$ the set of indices
$j$ such that for any $k\geq j$, $x_k=0$. By Lemma \ref{lemma4}:
\begin{enumerate}
\item If $j\in J$, for any $i\geq 1$, $\lambda_{i,j}=1-i$.
\item If $j\in K$, for any $i\geq 1$, $\lambda_{i,j}=0$.
\item If $j\notin J\cup K$, for any $i\geq 1$, $\lambda_{i,j}=1$. 
\end{enumerate}

\begin{lemma}\label{lemma7}
\begin{enumerate}
\item If $j,k\in J$ and $j\neq k$, then $j+k\in J$.
\item If $j\in J$ and $k\notin J\cup K$, then $j+k\notin J\cup K$.
\end{enumerate}
\end{lemma}

\begin{proof}
1. For any $i\geq 1$, by the pre-Lie relation $P(i,j,k)$:
\[(1-i)(1-i-j)-(1-j)\lambda_{i,j+k}=(1-i)(1-i-k)-(1-k)\lambda_{i,j+k}.\]
Hence, $(j-k)\lambda_{i,j+k}=(j-k)(1-i)$. As $j\neq k$, $\lambda_{i,j+k}=1-i$, so $j+k\in J$.\\

2. For any $i\geq 1$, by the pre-Lie relation $P(i,j,k)$:
\[1-i-\lambda_{i,j+k}=1-k-i-(1-k)\lambda_{i,j+k},\]
so $\lambda_{i,j+k}=1$: $j+k\notin J\cup K$. \end{proof}

\begin{lemma}
$J=\emptyset$ or $K=\emptyset$.
\end{lemma}

\begin{proof}
Let us assume that $J$ and $K$ are nonempty. Let us take $j\in J$ and let $k$ such that the smallest element of $K$ is $k+1$.
Then $k\notin J\cup K$, and the pre-Lie relation $PL(i,j,k)$ gives, for any $i\geq 1$:
\[1-i=1-i-k.\]
So $k=0$, which is absurd. So $J=\emptyset$ or $K=\emptyset$.
\end{proof}

\begin{lemma}
If $|J|\geq 2$, then there exists $m\geq 2$ such that $J=m\mathbb{N}^*$.
\end{lemma}

\begin{proof}
Let $m=\min(J)$. For any $m'\in J$, different from $m$, $m+(m'-m)=m'\in J$; by Lemma \ref{lemma7},
$m'-m\in J$. Inducing, for any $k\in \N$ such that $m'-km>0$, $m'-km\in J$.
In particular, if the Euclidean division of $m'$ by $m$ is $m'=qm+r$, then $r=0$ or $r\in J$.
As $r<m$, if $r>0$ this contradicts the definition of $m$, so $r=0$. We proved that $J\subseteq m\mathbb{N}^*$. \\

Let $m'\in J$, different from $m$. There exists $k>1$, such that $m'=km$. We proved that $m'-(k-2)m=2m\in J$.
By Lemma \ref{lemma7}, $m+2m=3m\in J$; an easy induction proves that $lm\in J$ for any $l\geq 1$.
So $J=m\mathbb{N}^*$.
\end{proof}

These lemmas give the following result:

\begin{prop}\label{prop seq 01}
The elements of $\seq_{0-1}$ are the following:
\begin{itemize}
\item For any $i,j \geq 1$, $\lambda_{i,j}=1$.
\item Case $A(m)$: there exists $m\geq 2$ such that $\displaystyle \lambda_{i,j}=\begin{cases}
1\mbox{ if }j<m,\\
0\mbox{ otherwise}.
\end{cases}$
\item Case $B(m)$: there exists $m\geq 2$ such that $\displaystyle \lambda_{i,j}=\begin{cases}
1-i\mbox{ if }m\mid j,\\
1\mbox{ otherwise}.
\end{cases}$
\item Case $C(m)$: there exists $m\geq 2$ such that $\displaystyle \lambda_{i,j}=\begin{cases}
1-i\mbox{ if }j=m,\\
1\mbox{ otherwise}.
\end{cases}$
\end{itemize}
\end{prop}

In terms of the rightmost diagonal $(x_j)_{j \geq 1}$, the proposition says that $(x_j)_{j \geq 1}$ must be all $1$'s, finitely many $1$'s followed by an infinite string of $0$'s (case $A(m)$), finitely many $1$'s followed by a $0$ and repeated (Case $B(m)$), or a string of all $1$'s with a single $0$. See Figure \ref{fig::seq01prop}.

\begin{figure} 
    \centering
    \begin{subfigure}{1\textwidth}
     \centering
     \setlength{\tabcolsep}{6pt}
     \begin{tabular}{ccccccccccccccc}
 & & & & & & & $1$ & & & & & & &  \\ 
 & & & & & & $1$ & &$1$ & & & & & & \\  
 & & & & & $1$ & & $1$ & & $1$ & & & & & \\
 & & & & $1$ & &$1$ & & $1$ & & $1$& & & & \\
 & & & $1$ & & $1$ & & $1$& &  $1$& & $1$ & & & \\
 & & $1$ & & $1$  & &$1$& & $1$& & $1$ & & $1$& & \\
 & $1$ & & $1$ & & $1$  & & $1$& &$1$ & &$1$ & & $1$ &\\
 $1$ & & $1$ & & $1$ & & $1$ & & $1$ & & $1$ & & $1$ & & $1$
    \end{tabular}
    \caption{The $\Lambda$-array corresponding to ladders.}
    \label{fig::ladder_lambda_array}
     \end{subfigure}
     \begin{subfigure}{1\textwidth}
     \setlength{\tabcolsep}{6pt}
     \centering
     \begin{tabular}{ccccccccccccccc}
 & & & & & & & $1$ & & & & & & &  \\ 
 & & & & & &$1$ & & $1$ & & & & & & \\  
 & & & & & $1$ & & $1$ & & $0$ & & & & & \\
 & & & & $1$ & & $1$ & & $0$ & &$0$& & & & \\
 & & & $1$ & & $1$ & & $0$& &  $0$& &$0$& & & \\
 & &$1$ & & $1$  & & $0$ & & $0$ & & $0$ & & $0$& & \\
 & $1$ & & $1$ & & $0$  & & $0$& & $0$ & & $0$ & & $0$ &\\
 $1$ & & $1$ & & $0$& & $0$  & & $0$ & & $0$ & & $0$ & &$0$
\end{tabular}
    \caption{$A(3)$.}
    \label{fig::A(3)}
     \end{subfigure}
     \begin{subfigure}{1\textwidth}
     \setlength{\tabcolsep}{6pt}
     \centering
     \begin{tabular}{ccccccccccccccc}
 & & & & & & & $1$ & & & & & & &  \\ 
 & & & & & & $1$ & & $1$& & & & & & \\  
 & & & & & $1$ & & $1$ & & $0$ & & & & & \\
 & & & & $1$ & & $1$ & & $-1$ & & $1$& & & & \\
 & & & $1$ & & $1$ & & $-2$& &  $1$& & $1$ & & & \\
 & & $1$ & &$1$ & &$-3$ & & $1$ & & $1$ & & $0$ & & \\
 & $1$ & & $1$ & & $-4$ & & $1$& & $1$ & & $-1$ & & $1$ &\\
 $1$ & & $1$ & & $-5$ & & $1$ & & $1$ & & $-2$ & & $1$  & & \hspace{5mm}$1$
\end{tabular}
    \caption{$B(3)$.}
    \label{fig::B(3)}
     \end{subfigure}
     \begin{subfigure}{1\textwidth}
     \setlength{\tabcolsep}{6pt}
     \centering
     \begin{tabular}{ccccccccccccccc}
 & & & & & & & $1$ & & & & & & &  \\ 
 & & & & & & $1$ & & $1$ & & & & & & \\  
 & & & & & $1$ & & $1$ & & $0$ & & & & & \\
 & & & & $1$ & & $1$ & & $-1$ & & $1$& & & & \\
 & & & $1$ & & $1$& & $-2$& &  $1$& & $1$ & & & \\
 & & $1$& & $1$ & &$-3$ & & $1$ & & $1$ & & $1$ & & \\
 & $1$ & & $1$ & & $-4$ & & $1$& & $1$ & & $1$ & & $1$ &\\
 $1$ & & $1$ & & $-5$ & & $1$  & & $1$ & & $1$ & & $1$ & & $1$
\end{tabular}
    \caption{$C(3)$.}
    \label{fig::C(3)}
     \end{subfigure}
    \caption{The four cases of Proposition \ref{prop seq 01} when $m = 3$.}
    \label{fig::seq01prop}
\end{figure}

\begin{prop}
In the case $B(m)$ or $C(m)$, let us denote by $a_n$ the coefficient of $B^+(\tun^n)$ in $t_{n+1}$. The generating
formal series of these coefficients is:
\[G(X)=1+\sum_{n\geq 1}a_nX^n=\frac{1+X}{(1-(-X)^m)^{\frac{1}{m}}}.\]
Consequently, for any $n\geq 0$:
\begin{align*}
a_n&=\begin{cases}
Z(\mathfrak{S}_n,X_1,\ldots,X_n)\Bigg|_{ X_i=\begin{cases}
\mbox{\scriptsize $(-1)^i$  if $m\mid i$,}\\
\mbox{\scriptsize $0$ otherwise}
\end{cases}} & \mbox{ if $m \mid n$},\\
a_{n-1} & \mbox{ if $m \mid n+1$},\\
0 & \mbox{ otherwise.}
\end{cases}\end{align*}
where $Z$ is the cycle index polynomial; see Section 5.2 of \cite{StanleyVolume2}.
\end{prop}

\begin{proof}
For any $i\geq 1$, $j\geq 0$, we put 
\[a_{i,j}=\lambda(\underbrace{1,\ldots,1}_{\mbox{\scriptsize $j$ times}},i).\]
Then, for any $i,j\geq 1$:
\[a_{i,j}=a_{1,j-1}\lambda_{j+1,i}-\sum_{k=1}^j a_{i+1,j-1}\lambda_{1,i}
=a_{1,j-1}\lambda_{j+1,i}-j\lambda_{1,i}a_{i+1,j-1}.\]
Hence:
\begin{itemize}
\item If $i<m$, $a_{i,0}=1$ and $a_{i,j}=a_{1,j-1}-ja_{i+1,j-1}$ for any $j\geq 1$.
\item If $i=m$, $a_{i,0}=1$ and $a_{i,j}=-ja_{1,j-1}$.
\end{itemize}
If $i\geq 1$, we put:
\[F_i=\sum_{j\geq 0} \frac{a_{i,j}}{j!}X^j.\]
Then:
\[G=1+\sum_{n=1}^\infty \frac{a_{1,n-1}}{n!}X^n=1+\int_0^X F_1(t)\mathrm{d}t.\]
The preceding relations give:
\begin{itemize}
\item If $i<m$:
\begin{align*}
F_i&=1+\sum_{j=1}^\infty \frac{a_{1,j-1}}{j!}X^j-\sum_{j=1}^\infty a_{i+1,j-1}{(j-1)!}X^j=G-XF_{i+1}.
\end{align*}
\item If $i=m$:
\begin{align*}
F_i&=-\sum_{j=1}^\infty \frac{a_{1,j-1}}{(j-1)!}X^j=-XF_1.
\end{align*}
\end{itemize}
A direct induction on $i$ proves that for any $i<m$:
\[F_1=\frac{1-(-X)^i}{1+X}G+(-X)^iF_{i+1}.\]
In particular, for $i=m$:
\[F_1=\frac{1-(-X)^{m-1}}{1+X}G+(-X)^mF_1.\]
Hence, $G$ is the unique solution of:
\[\begin{cases}
(1-(-X)^m)G'=\displaystyle \frac{1-(-X)^{m-1}}{1+X}G,\\
G(0)=1.
\end{cases}\]
This gives:
\[G=\frac{1+X}{(1-(-X)^m)^{\frac{1}{m}}}.\]
Consequently:
\begin{align*}
\ln(G)&=\ln(1+X)-\frac{1}{m}\ln(1-(-X)^m)\\
&=\sum_{k=1}^\infty \frac{(-1)^{k+1}}{k}X^k+\sum_{k=1}^\infty \frac{(-1)^{mk}}{mk}X^{mk}\\
&=\sum_{k\notin m\N} \frac{(-1)^{k+1}}{k}X^k.
\end{align*}
On the other side, for any $n\geq 0$:
\[Z(\mathfrak{S}_n,X_1,\ldots,X_n)=\sum_{i_1+2i_2+\ldots+ni_n=n}\frac{1}{1^{i_1}\ldots n^{i_n}i_1!\ldots i_n!}
X^{i_1}\ldots X_n^{i_n}.\]
Hence:
\begin{align*}
\sum_{n=0}^\infty Z(\mathfrak{S}_n,X)&=\prod_{k=1}^\infty \sum_{j=0}^\infty \frac{X_i^j}{i^j j!}\\
&=\prod_{i=1}^\infty \exp\left(\frac{X_i}{i}\right)\\
&=\exp\left(\sum_{i=1}^\infty \frac{X_i}{i}\right).
\end{align*}
Hence:
\begin{align*}
\ln\left(\sum_{n=0}^\infty Z(\mathfrak{S}_n,X)\right)\Bigg|_{ X_i=\begin{cases}
\mbox{\scriptsize $0$  if $m\mid i$,}\\
\mbox{\scriptsize $(-1)^{i+1}X^i$ otherwise}
\end{cases}}
&=\sum_{i\notin m\N} \frac{(-1)^{i+1}}{i}X^i=\ln(G).
\end{align*}
Taking the exponential, we obtain that
\[G=\sum_{n=0}^\infty Z(\mathfrak{S}_n,X)\Bigg|_{ X_i=\begin{cases}
\mbox{\scriptsize $0$  if $m\mid i$,}\\
\mbox{\scriptsize $(-1)^{i+1}X^i$ otherwise}.
\end{cases}}\]
Considering the coefficient of $X^n$ gives:
\[a_n=Z(\mathfrak{S}_n,X_1,\ldots,X_n)\Bigg|_{ X_i=\begin{cases}
\mbox{\scriptsize $0$  if $m\mid i$,}\\
\mbox{\scriptsize $(-1)^{i+1}X^i$ otherwise}.
\end{cases}}\]
Moreover:
\begin{align*}
G&=\exp\left(\sum_{i\notin m\N} \frac{(-1)^{i+1}}{i}X^i\right)\\
&=\exp\left(\sum_{i\geq 1}\frac{((-1)^{i+1}}{i}X^i\right)\exp\left(\sum_{m\mid i}\frac{(-1)^{i}}{i}X^i\right)\\
&=(1+X)\sum_{n\geq 0}Z(\mathfrak{S}_n,X_1,\ldots,X_n)\Bigg|_{ X_i=\begin{cases}
\mbox{\scriptsize $(-1)^{mi}$  if $m\mid i$,}\\
\mbox{\scriptsize $0$ otherwise}.
\end{cases}}
\end{align*}
Considering the coefficient of $X^n$, this gives the formula for $a_n$, with the observation that
\[Z(\mathfrak{S}_n,X_1,\ldots,X_n)\Bigg|_{ X_i=\begin{cases}
\mbox{\scriptsize $(-1)^{mi}$  if $m\mid i$,}\\
\mbox{\scriptsize $0$ otherwise}
\end{cases}}=0\mbox{ if }n\notin m\N.\]
 \end{proof}

\section{Acknowledgements}
Parts of this work first appeared in the MMath thesis of the first author, \cite{Dmmath}.  KY owes a particular dept to Spencer Bloch and Dirk Kreimer for some conversations during a 2014 visit which inspired this notion of generalized renormalization group equation.  KY is supported by an NSERC Discovery grant and the Canada Research Chairs program.
LF acknowledges support from the grant ANR-20-CE40-0007.

\bibliographystyle{amsplain}
\bibliography{biblio}

\providecommand{\bysame}{\leavevmode\hbox to3em{\hrulefill}\thinspace}
\providecommand{\MR}{\relax\ifhmode\unskip\space\fi MR }
\providecommand{\MRhref}[2]{%
  \href{http://www.ams.org/mathscinet-getitem?mr=#1}{#2}
}
\providecommand{\href}[2]{#2}
\begin{thebibliography}{10}

\bibitem{Bms}
Paul-Hermann Balduf, \emph{Dyson-{S}chwinger equations in minimal subtraction},
  arXiv:2109.13684.

\bibitem{Bnote}
Spencer Bloch, \emph{Personal communication}.

\bibitem{Chapoton}
Fr\'ed\'eric Chapoton and Muriel Livernet, \emph{Pre-{L}ie algebras and the
  rooted trees operad}, Internat. Math. Res. Notices (2001), no.~8, 395--408.

\bibitem{ConnesKreimer}
Alain Connes and Dirk Kreimer, \emph{Hopf algebras, renormalization and
  noncommutative geometry}, Comm. Math. Phys. \textbf{199} (1998), no.~1,
  203--242.

\bibitem{Sorin_D_Hopf_Algebras_An_Introduction}
Sorin D\u{a}sc\u{a}lescu, Constantin N\u{a}st\u{a}sescu, and \c{S}erban Raianu,
  \emph{Hopf algebras: An introduction}, Pure and Applied Mathematics---A
  Series of Monographs and Textbooks, no. 235, Marcel Dekker, 2001.

\bibitem{Dmmath}
William Dugan, \emph{Sequences of trees and higher-order renormalization group
  equations}, {MM}ath, University of Waterloo, 2019.

\bibitem{Foissycontrole}
Lo\"ic Foissy, \emph{Cofree com-prelie bialgebras}, arxiv 1802.07642.

\bibitem{FoissyDyson}
\bysame, \emph{Fa\`a di {B}runo subalgebras of the {H}opf algebra of planar
  trees from combinatorial {D}yson-{S}chwinger equations}, Adv. Math.
  \textbf{218} (2008), no.~1, 136--162.

\bibitem{Fsys}
Lo\"ic Foissy, \emph{Classification of systems of {D}yson-{S}chwinger equations
  in the {H}opf algebra of decorated rooted trees}, Advances in Mathematics
  \textbf{224} (2010), no.~5, 2094--2150, arXiv:0909.0358.

\bibitem{Grossman}
Robert~L. Grossman and Richard~G. Larson, \emph{Differential algebra structures
  on families of trees}, Adv. in Appl. Math. \textbf{35} (2005), no.~1,
  97--119.

\bibitem{Hoffman}
Michael~E. Hoffman, \emph{Combinatorics of rooted trees and {H}opf algebras},
  Trans. Amer. Math. Soc. \textbf{355} (2003), no.~9, 3795--3811.

\bibitem{Lskript}
Dirk Kreimer, \emph{Renormalization \& renormalization group}, Scribed by Lutz
  Klaczynski,
  https://www2.mathematik.hu-berlin.de/~kreimer/wp-content/uploads/SkriptRGE.pdf.

\bibitem{Kreimer1997}
\bysame, \emph{On the hopf algebra structure of perturbative quantum field
  theories}, Advances in Theoretical and Mathematical Physics \textbf{2}
  (1997), 303--334.

\bibitem{OudomGuin}
J.-M. Oudom and D.~Guin, \emph{On the {L}ie enveloping algebra of a pre-{L}ie
  algebra}, J. K-Theory \textbf{2} (2008), no.~1, 147--167.

\bibitem{Panaite}
Florin Panaite, \emph{Relating the {C}onnes-{K}reimer and {G}rossman-{L}arson
  {H}opf algebras built on rooted trees}, Lett. Math. Phys. \textbf{51} (2000),
  no.~3, 211--219.

\bibitem{StanleyVolume2}
Richard~P. Stanley, \emph{Enumerative combinatorics volume 2}, 1 ed., Cambridge
  University Press, 1999.

\bibitem{StanleyVolume1}
\bysame, \emph{Enumerative combinatorics volume 1}, 2 ed., Cambridge University
  Press, 2012.

\end{thebibliography}

\end{document}